\documentclass[a4paper,10pt]{article}
\usepackage{graphicx}
\usepackage{subfigure}
\usepackage{epstopdf}
\usepackage{booktabs}
\usepackage{threeparttable}
\usepackage{amssymb, amsmath, amsthm, epsfig}
\allowdisplaybreaks
\numberwithin{equation}{section}
\usepackage{latexsym, enumerate}
\usepackage{eepic}
\usepackage{epic}
\usepackage{color}
\newcommand{\OneB}[1]{{\color{red}{#1}}}
\newcommand{\red}[1]{\textcolor[rgb]{1.00,0.00,0.00}{#1}}
\newcommand{\green}[1]{\textcolor[rgb]{0.00,0.65,0.00}{#1}}
\newcommand{\blue}[1]{\textcolor[rgb]{0.00,0.00,1.00}{#1}}
\usepackage{amsmath,textcomp, amssymb, mathrsfs, bm, amsthm, amsfonts}
\usepackage{algorithm}
\usepackage{algpseudocode}
\usepackage{verbatim}
\allowdisplaybreaks[4]
\usepackage{ifpdf}
\usepackage[hidelinks]{hyperref}
\usepackage[numbers]{natbib}

\usepackage{chngcntr}

\usepackage{dsfont}
\usepackage{multirow}
\usepackage{bm}
\usepackage{caption}

\theoremstyle{plain}
\newtheorem{lem}{Lemma}[section]
\newtheorem{thm}[lem]{Theorem}

\theoremstyle{definition}
\newtheorem{defn}{Definition}[section]

\theoremstyle{remark}
\newtheorem{rem}{Remark}[section]

\renewcommand{\theequation}{\thesection.\arabic{equation}}
\renewcommand{\thefigure}{\thesection.\arabic{figure}}

\newcommand{\sst}{\scriptstyle}
\newcommand{\intl}{\int\limits}
\newcommand{\ovl}{\overline}
\newcommand{\n}{\noindent}
\newcommand{\ds}{\displaystyle}
\newcommand{\vp}{\varepsilon}
\newcommand{\wt}{\widetilde}
\newcommand{\ms}{\medskip}
\begin{document}
	\renewcommand{\figurename}{Figure}
	\renewcommand{\thesubfigure}{(\alph{subfigure})}
	\renewcommand{\thesubtable}{(\alph{subtable})}
	\makeatletter
	\renewcommand{\p@subfigure}{\thefigure~}
	\bibliographystyle{plainnat}
	\makeatother
	\title{\large\bf Inverse Random Source Problem for the Helmholtz Equation from Statistical Phaseless Data}
	\author{
		Qiao-Ping Chen\thanks
		{School of Mathematics, Hunan University, Changsha 410082, China.
			Email: s230600375@hnu.edu.cn}
\and
Hongyu Liu\thanks{Department of Mathematics, City University of Hong Kong, Kowloon, Hong Kong, China.\ \ Email: hongyu.liuip@gmail.com; hongyliu@cityu.edu.hk}
		\and
		Zejun Sun\thanks
		{School of Mathematics, Hunan University, Changsha 410082, China. Email: sunzejun@hnu.edu.cn}
		\and
		Li-Li Wang\thanks
		{School of Mathematics, Hunan University, Changsha 410082, China.
			Email: lilywang@hnu.edu.cn}
		\and
		Guang-Hui Zheng\thanks
		{School of Mathematics, Hunan University, Changsha 410082, China. Email: zhenggh2012@hnu.edu.cn; zhgh1980@163.com}
	}
	\date{}
	\maketitle
	\begin{center}{\bf ABSTRACT}
	\end{center}\smallskip
	This paper investigates the problem of reconstructing a random source from statistical phaseless data for the two-dimensional Helmholtz equation. The major challenge of this problem is non-uniqueness, which we overcome through a reference source technique. Firstly, we introduce some artificially added point sources into the inverse random source system and derive phase retrieval (PR) formulas for the expectation and variance of the radiated fields. This paper rigorously analyze the uniqueness and stability of the recovered statistics of the radiated fields. Afterwards, since the direct problem has a unique mild solution, by examining the expectation and variance of this solution and combined with the phase retrieval formulas, we derive the Fredholm integral equations to solve the inverse random source problem (IRSP). We prove the stability of the corresponding integral equations. To quantify the uncertainty of the random source, we utilize the Bayesian method to reconstruct the random source and establish the well-posedness of the posterior distribution. Finally, numerical experiments demonstrate the effectiveness of the proposed method and validate the theoretical results.
	
	\smallskip
	{\bf keywords}: Stochastic Helmholtz equations; Inverse random source problem; Phase retrieval;  Bayesian inference
	\section{Introduction}
	The inverse source problem, which involves reconstructing an unknown source from measurements of the radiated fields, has broad applications in scientific and engineering areas, such as acoustic tomography \textcolor{blue}{\cite{Liu2017,JLi2009}}, seismic exploration \textcolor{blue}{\cite{Weglein2003,Bube1983}}, medical imaging \textcolor{blue}{\cite{Albanese2006,Deng2017}}, and antenna design \textcolor{blue}{\cite{Palmeri2018,Xiao2021}}. Existing studies of the inverse source problem mostly consider that the source is a deterministic function \textcolor{blue}{\cite{ElBadia2000,Bleistein1977,Zhang2015}}. However, in many practical situations, the source is better modeled as a random field. Compared to deterministic inverse source problems, inverse random source problems (IRSP) involve inherent randomness and uncertainties, which make the problem more difficult. In this case, the solution of the forward problem is a random function. The statistics of the solution, such as expectation, variance and higher-order moments, play a more significant role. Recent studies have proposed various theoretical and numerical methods for addressing IRSP. In particular, the IRSP associated with the Helmholtz equation or time-harmonic acoustic model is addressed in this paper, which refers to \textcolor{blue}{\cite{Bao2014,Bao2016,Caro2019,Devaney1979,Lassas2008,Li2020,Liliu2021,Li2019,Li2021}}.
%

	In existing research, most studies on reconstructing sources rely on complete datasets, including both intensity and phase information. However, acquiring the full data is often unfeasible or even impossible. Typically, only the intensity or modulus of the field can be measured, espectially in high-frequency regime. This limitation has motivated the research for inverse source problems with phaseless data \textcolor{blue}{\cite{Chang2024,Zhang2018Retri,Ji2019}}. Xu et al. \textcolor{blue}{\cite{Xu2018}} is concerned with uniqueness in inverse acoustic scattering with phaseless far-field data at a fixed frequency.  O. Ivanyshyn and R. Kress \textcolor{blue}{\cite{Ivanyshyn2011}} propose a modification of the Newton-type algorithm to recover a real-valued surface impedance from phase-less far field data. Inverse scattering problems using phaseless data have been widely examined by many researchers \textcolor{blue}{\cite{Ammari2016,Klibanov2014a,Klibanov2017a,Klibanov,ZZ2017,zz2018}}.
	
	Considering IRSP when only phaseless data is available, the problem becomes even more challenging. The reconstruction process is more complex, and the theoretical analysis of uniqueness and stability encounters greater difficulties. It is precisely for this reason that related results are nearly absent.
	
	This paper is concerned with the problem of reconstructing a random acoustic source in the two-dimensional Helmholtz equation from multi-frequency phaseless data. More precisely, consider the scattering problem of the Helmholtz equation in a homogeneous medium
	\begin{equation}\label{Helmholtz}
		\left\{
		\begin{aligned}
			\Delta u + k^2 u &= f \quad \text{in} \ \mathbb{R}^2,  \\
			\frac{\partial u}{\partial r} - \mathrm{i}k u &= o(r^{-1/2}), \quad r = |x| \to \infty,
		\end{aligned}
		\right.
	\end{equation}
	where the wavenumber $k>0$ is a constant, $f$ is the electric current density, and $u$ is the radiated random wave field that satisfies the Sommerfeld radiation condition.
	Here, \( f \) is assumed to be a random function driven by an additive noise, given by
	\begin{equation}
		f(x) = g(x) + \sigma(x) \dot{W}_x,
	\end{equation}
	where \( g \) and \( \sigma \geq 0 \) are two deterministic real functions with compact supports contained in the rectangular domain \( \Omega_0 \subset \mathbb{R}^2 \), and \( \dot{W}_x \) represents the white noise. In this random source model, the mean, standard deviation, and variance of $f$ are $g$, $\sigma$, and $\sigma^2$, respectively. We assume that $g\in L^2(\Omega_0)$ and $\sigma\in L^4(\Omega_0)$. Then, there exists a unique continuous stochastic process $u$ satisfying
	\begin{equation}\label{solve}
		u(x,k)=\int_{\Omega_0}G_k(x,y)g(y)\:\mathrm{d}y+
		\int_{\Omega_0}G_k(x,y)\sigma(y)\:\mathrm{d}W_y,
	\end{equation}
	which is referred to as the mild solution of the stochastic scattering problem \textcolor{blue}{\eqref{Helmholtz}} [\textcolor{blue}{\citealp{Bao2016}}, Theorem 2.7].
	Here  $G_k$ is Green's function of the Helmholtz equation. Specifically, we have
	$$G_k(x,y):=-\frac{\mathrm{i}}4H_0^{(1)}(k|x-y|),$$
	where $H_{0}^{(1)}$ is the Hankel function of the first kind of order zero.
	
	In the following, the notation \( u(x,k) \) is used to indicate the explicit dependence of the radiated field \( u(x) \)  on the wavenumber \( k \).
Suppose that the radiated fields are measured on \( \partial B_R := \{x \in \mathbb{R}^2 : |x| = R\} \), with \( \Omega_0 \subset\subset B_R := \{x \in \mathbb{R}^2 : |x| < R\} \). Then the IRSP considered in this paper is described as follows:
\\

\textbf{(IRSP)}\ \ Reconstruct the source \( f(x) \) from the multi-frequency phaseless statistics data \( \{\mathbf{E}(|u(x,k)|^2);\\ \mathbf{Var}(|u(x,k)|^2)\}: x \in \partial B_R,\, k \in \mathbb{K}_N \} \), where \( N \in \mathbb{N}_+ \) and \( \mathbb{K}_N \) is an admissible set consisting of a finite number of wavenumbers.
\\

	This problem is challenging due to its inherent non-uniqueness. Specifically, the source term $f(x)$ and its negative counterpart $-f(x)$ generate identical phaseless data, leading to ambiguity in source recovery.
	
	To address this issue, motivated by reference scatterer techniques that are applied to inverse scattering problems \textcolor{blue} {\cite{JLi2009,Novikov2015,Novikov2015b,Zhang2018Uni,Zhang2018Retri}}, we develop a reference source technique by using some artificially added point sources as additional reference sources to retrieve the phase information. By adding these point sources, we derive two systems of equations, one system for recovering the expectation and the other for recovering the variance and the covariance. The point sources are placed at suitable locations, which guarantees that the absolute value of the determinant of the coefficient matrix for each of the two systems of equations has a strictly positive lower bound. Thereby, the two systems are ensured to be uniquely solvable, leading to the stability results for the phase information of both the mean and variance of the radiated fields on $\partial B_R$.
	
	Once the phase retrieval formulas are obtained, the problem becomes a standard IRSP. The inverse problem is to reconstruct the random source by utilizing the measured phaseless statistical information together with the phase retrieval formulas. Given the random source, the direct problem is to determine the random wave field. The authors have demonstrated that a unique mild solution exists for the stochastic direct scattering problem in \textcolor{blue}{\cite{Bao2016}}. By analyzing the statistics of the mild solution and incorporating the phase retrieval formulas, we derive Fredholm integral equations to solve the IRSP. We have established the stability results for the phaseless IRSP. To quantify the uncertainty of the random source, we employ the Bayesian method to address the IRSP and have shown the well-posedness of the posterior distribution.

	The major contributions of this work can be summarized as follows:
	\begin{itemize}
		\item
		We investigate a class of inverse problems that utilize phaseless statistical data to reconstruct two-dimensional random wave sources. To date, there has been relatively little research compared to the inverse problems involving phase data. Specifically, by introducing reference point sources, we ingeniously integrate phase retrieval techniques, integral equation theory, and Bayesian inference methods to address the challenges posed by inherent ill-posedness and randomness of inverse problems. Consequently, we develop effective numerical algorithms for IRSP.
		
		\item
		We prove the stability estimates for the phase retrieval formulas and the corresponding integral equations, and we establish the well-posedness of the posterior measure within the Bayesian inference framework. This provides a theoretical foundation and guidance for the subsequent design of effective reconstruction algorithms.
		
		\item
		Based on phase retrieval techniques, integral equation theory, and Bayesian inference methods, we have constructed several reconstruction formulas for random wave sources within the Bayesian framework, relying on different artificial point sources. Finally, we employed a dimension-independent Preconditioned Crank-Nicolson (PCN) sampling algorithm to validate the effectiveness and robustness of these reconstruction formulas against measurement errors.

	\end{itemize}
	
	The outline of this paper is as follows. In Section \textcolor{blue}{\ref{2}}, we introduce the reference source technique and derive phase retrieval formulas for the mean and variance of the radiated fields. We also prove the unique solvability of the phase retrieval formulas and the stability of the recovered statistics. Section \textcolor{blue}{\ref{3}} is dedicated to the stochastic inverse problem, where we derive Fredholm integral equations. The stability of integral equations are established. In Section \textcolor{blue}{\ref{4}}, we
	apply the Bayesian method to solve the inverse problem and demonstrate the well-posedness of the posterior distribution. Section \textcolor{blue}{\ref{5}} verifies the phase retrieval formulas through numerical experiments, and evaluates the performance of the Bayesian method in reconstructing the random source. Finally, conclusions are provided in Section \textcolor{blue}{\ref{6}}.
	
	\section{Phase retrieval formulas}\label{2}
	In this section, we mainly introduce how to derive two systems of equations for the phase retrieval formulas, one for the expectation and the other for the variance and the covariance, by incorporating artificially added point sources.
	
	Compared to the literature \textcolor{blue}{\cite{Zhang2018Retri}}, which uses the reference source technique to invert a deterministic source from phaseless data, our study discusses the IRSP with phaseless data.
	Some necessary notation and definitions are introduced. Without loss of generality, let
	$$\Omega_0 = (-a,a)^2, \quad a>0.$$
	We introduce the definition of an admissible set, which is slightly modified from  \textcolor{blue}{\cite{Zhang2018Retri}}.
	\begin{defn}
		Let $N\in\mathbb{N}_+$ and $k^*\in\mathbb{R}_+$ be a small wavenumber such that
		$0<k^*R<1.$ Then the admissible set of wavenumbers is given by
		$$\mathbb{K}_N:=\left\{\frac\pi ax, x\geqslant 1, x \in \mathbb{R}\right\}\bigcup\left\{k^*\right\}.$$
	\end{defn}
	Let $m$ be a positive integer and $\nu_s=s \pi / m$ for $s \in \mathbb{R}$. Denote
	
	$$
	\begin{aligned}
		& B_j:=\left\{r(\cos \theta, \sin \theta): 0 \leqslant r \leqslant R, \nu_{2 j-2} \leqslant \theta \leqslant \nu_{2 j}\right\}, j=1, \cdots, m, \\
		& \partial B_j:=\left\{R(\cos \theta, \sin \theta): \nu_{2 j-2} \leqslant \theta \leqslant \nu_{2j}\right\}, j=1, \cdots, m .
	\end{aligned}
	$$
	For each $j\ (j=1, \cdots, m)$, take two points $z_{j,\ell}:=\lambda_{j,\ell} R\left(\cos \nu_{2 j-1}, \sin \nu_{2 j-1}\right)$, where $\lambda_{j,\ell} \in \mathbb{R}, \sqrt{2} a / R \leqslant\left|\lambda_{j,\ell}\right|<1$ and $\ell=1,2$.
	Then, we have
	
	$$
	B_R=\bigcup_{j=1}^m B_j  \text { and }\partial B_R=\bigcup_{j=1}^m \partial B_j.
	$$
	The geometrical configuration is illustrated in Figure \textcolor{blue}{\ref{fig1}}.

	\begin{figure}[ht]
		\centering
		\includegraphics[width=0.6\linewidth]{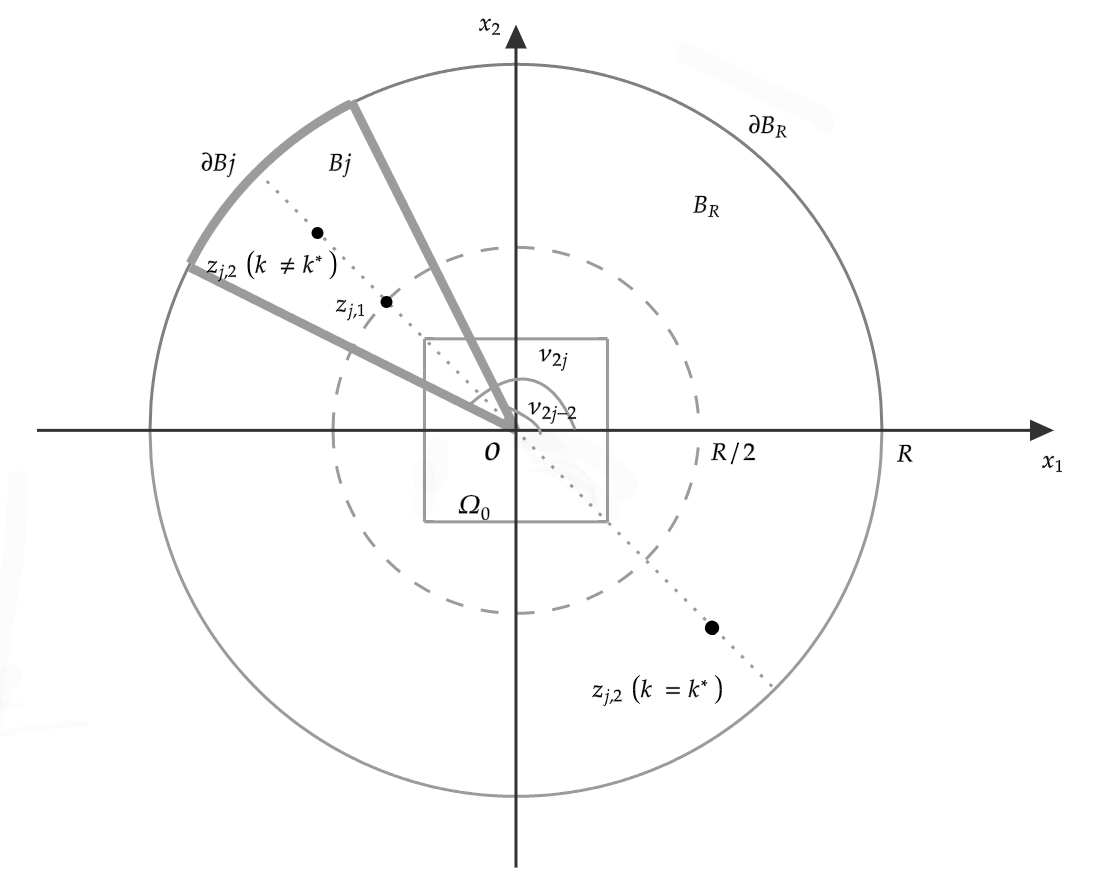}
		\captionsetup{justification=justified, labelsep=period}  
		\caption{An illustration of the geometrical setting for the reference source technique.
			The locations of $z_{j,\ell}$, $\ell$= 1, 2, are marked by small black points.}
		\label{fig1}
	\end{figure}
	
	Let $\delta_{j,\ell}(j=1, \cdots, m, \ell=1,2)$ denote the Dirac distributions at the points $z_{j,\ell}$ and define $\|\cdot\|_{j,\infty}:=\|\cdot\|_{L^\infty(\partial B_j)}.$
	$G_k(x,z_{j,\ell})$ satisfies the following inhomogeneous Helmholtz equation
	$$\Delta G_{j,\ell}+k^2 G _{j,\ell}=\delta_{j,\ell}\quad\mathrm{~in~}\mathbb{R}^2.$$
	Since $H_0^{(1)}(z)=J_0(z)+\mathrm{i} Y_0(z)$ for $z \in \mathbb{C} \backslash\{0\}$, we have
	
	$$
	\operatorname{Re}\left(G_{j, \ell}(x, k)\right)=\frac{1}{4} Y_0\left(k r_{j, \ell}\right), \quad \operatorname{Im}\left(G_{j,\ell}(x, k)\right)=-\frac{1}{4} J_0\left(k r_{j, \ell}\right) ,
	$$
	where $r_{j,\ell}=|x-z_{j,\ell}|$, $J_0$ and $Y_0$ are the Bessel functions of the first and second kind with order zero, respectively. Further, $v_{j,\ell}:=u+G_{j,\ell}$ is the unique solution to the problem
	$$\begin{cases}\Delta v_{j,\ell}+k^2v_{j,\ell}=f+\delta_{j,\ell},&\quad\text{in }\mathbb{R}^2,\\\frac{\partial v_{j,\ell}}{\partial r}-\text{i}kv_{j,\ell}=o\left(r^{-1/2}\right),&\quad r=|x|\to\infty.\end{cases}$$
For simplicity, we denote
	$$
	u_{j,k}(\cdot):=\left.u(\cdot, k)\right|_{\partial B_j,} \quad v_{j, \ell,k}(\cdot):=\left.v_{j, \ell}(\cdot, k)\right|_{\partial B_j,} \quad \ell=1,2,
	$$
	for each fixed $j=1, \cdots, m$ and $k \in \mathbb{K}_{\mathrm{N}}$.
	With these preparations, we consider the following phase retrieval problem:
\\

\textbf{(PRP)}\ \ Given \( N \in \mathbb{N}_+, k \in \mathbb{K}_N \) and the phaseless statistical data

$\{\mathbf{E}^\epsilon(|u(x,k)|^2)$, $\mathbf{Var}^\epsilon(|u(x,k)|^2)\}$, $x \in \partial B_R$;

$\{\mathbf{E}^\epsilon(|v_{j,1,k}|^2)$, $\mathbf{E}^\epsilon(|v_{j,2,k}|^2)$, $\mathbf{Cov}^\epsilon(|v_{j,1,k}|^2, |v_{j,2,k}|^2)$, $\mathbf{Cov}^\epsilon(|u_{j,k}|^2, |v_{j,1,k}|^2)$, $\mathbf{Cov}^\epsilon(|u_{j,k}|^2,|v_{j,2,k}|^2)$, $\mathbf{Var}^\epsilon(|v_{j,1,k}|^2)$, $\mathbf{Var}^\epsilon(|v_{j,2,k}|^2)\}$ for $j=1,\cdots,m$;
	
Recover the mean of the radiated fields $\mathbf{E}(u_{j,k})$, the covariance $\mathbf{Cov}(\operatorname{Re}(u_{j,k}), \operatorname{Im}(u_{j,k}))$, and the variance: $\mathbf{Var}(\operatorname{Re}(u_{j,k}))$, $\mathbf{Var}(\operatorname{Im}(u_{j,k}))$ for $j=1,\cdots,m$ 	
\\

Now, we turn to the derivation of the phase retrieval formulas. Notice
	$$
	\operatorname{Re}\left(v_{j, \ell,k}\right)=\operatorname{Re}(u_{j,k})+\frac{1}{4} Y_0\left(k r_{j, \ell}\right), \quad \operatorname{Im}\left(v_{j,\ell,k}\right)=\operatorname{Im}(u_{j,k})-\frac{1}{4} J_0\left(k r_{j, \ell}\right) .
	$$
	Therefore, we have
	\begin{equation}\label{v_jlk}
		\begin{aligned}
			|v_{j,\ell,k}|^2 &=(\operatorname{Re}(u_{j,k})+\frac{1}{4} Y_0\left(k r_{j, \ell}\right))^2+(\operatorname{Im}(u_{j,k})-\frac{1}{4} J_0(k r_{j, \ell}))^2 \\
			&=(\operatorname{Re}(u_{j,k}))^2+(\operatorname{Im}(u_{j,k}))^2+\frac{1}{2}Y_0(kr_{j,\ell})\cdot \operatorname{Re}(u_{j,k}) -\frac{1}{2}J_0(kr_{j,\ell})\cdot \operatorname{Im}(u_{j,k})\\&\quad+\frac{1}{16}Y_0^2(kr_{j,\ell})+\frac{1}{16}J_0^2(kr_{j,\ell})\\
			&=|u_{j,k}|^2+\frac{1}{16}\left|H_0^{(1)}(kr_{j,\ell})\right|^2+\frac{1}{2}Y_0(kr_{j,\ell})\cdot \operatorname{Re}(u_{j,k}) -\frac{1}{2}J_0(kr_{j,\ell})\cdot\operatorname{Im}(u_{j,k}).\\
		\end{aligned}
	\end{equation}
	We have established a relationship between the phaseless and phase data, enabling us to derive the phase retrieval formulas.
	\subsection{Recovering expectation}
	In this section, we derive the phase retrieval formula for the expectation.
	Taking the expectation on both sides of \textcolor{blue}{\eqref{v_jlk}}, we have
	\begin{equation}
		\mathbf{E}(|v_{j,\ell,k}|^2) =\mathbf{E}(|u_{j,k}|^2)+\frac{1}{16}\left|H_0^{(1)}(kr_{j,\ell})\right|^2+\frac{1}{2}Y_0(kr_{j,\ell})\cdot \mathbf{E}(\operatorname{Re}(u_{j,k})) -\frac{1}{2}J_0(kr_{j,\ell})\cdot \mathbf{E}(\operatorname{Im}(u_{j,k})).\\
		\notag
	\end{equation}
	Hence,
	\begin{equation}\label{Ef}
		\begin{aligned}
			&Y_0(kr_{j,1})\cdot \mathbf{E}(\operatorname{Re}(u_{j,k})) -J_0(kr_{j,1})\cdot \mathbf{E}(\operatorname{Im}(u_{j,k})) = f_{j,1,k},\\
			&Y_0(kr_{j,2})\cdot \mathbf{E}(\operatorname{Re}(u_{j,k})) -J_0(kr_{j,2})\cdot \mathbf{E}(\operatorname{Im}(u_{j,k})) = f_{j,2,k},
		\end{aligned}
	\end{equation}
	where
	\begin{equation}\label{f}
		f_{j,\ell,k}=2\left(\mathbf{E}(|v_{j,\ell,k}|^{2})-\mathbf{E}(|u_{j,k}|^{2})\right)-\frac{1}{8}\left|H_{0}^{(1)}(kr_{j,\ell})\right|^{2},\quad\ell=1,2.
\end{equation}

Thus, by a direct calculation, the phase retrieval formula for the expectation on \(\partial B_j\) is obtained:
\begin{equation}\label{E}
	\mathbf{E}(\operatorname{Re}(u_{j,k})) = \frac{\operatorname{det}(A_{j,k}^R)}{\operatorname{det}(A_{j,k})},\quad \mathbf{E}(\operatorname{Im}(u_{j,k})) = \frac{\operatorname{det}(A_{j,k}^I)}{\operatorname{det}(A_{j,k})},
\end{equation}
where the function matrices $A_{j,k}, A_{j,k}^R, A_{j,k}^I$ are defined as follows
$$
\begin{aligned}&A_{j,k}=\begin{pmatrix}Y_0(kr_{j,1})&-J_0(kr_{j,1})\\Y_0(kr_{j,2})&-J_0(kr_{j,2})\end{pmatrix},\quad A_{j,k}^\mathrm{R}=\begin{pmatrix}f_{j,1,k}&-J_0(kr_{j,1})\\f_{j,2,k}&-J_0(kr_{j,2})\end{pmatrix},\\&A_{j,k}^\mathrm{I}=\begin{pmatrix}Y_0(kr_{j,1})&f_{j,1,k}\\Y_0(kr_{j,2})&f_{j,2,k}\end{pmatrix}.\end{aligned}$$
Therefore, the mean of the radiated fields $\mathbf{E}(u_{j,k})$ can be recovered from $\mathbf{E}(u_{j,k}) = \mathbf{E}(\operatorname{Re}(u_{j,k})) + i\mathbf{E}(\operatorname{Im}(u_{j,k}))$ for
$j = 1, \cdots, m.$

\subsection{Recovering variance and covariance}
In this section, we will derive the phase retrieval formula for the variance.

Based on \textcolor{blue}{\eqref{v_jlk}}, for each fixed $j=1, \cdots, m$ and $k \in \mathbb{K}_{\mathrm{N}}$, we have
\begin{equation}\label{covuv}
	\begin{aligned}
		&\mathbf{Cov}(|u_{j,k}|^2,|v_{j,\ell,k}|^2)\\
		=& \mathbf{Cov}(|u_{j,k}|^2,|u_{j,k}|^2+\frac{1}{2}Y_0(kr_{j,\ell})\cdot \operatorname{Re}(u_{j,k}) -\frac{1}{2}J_0(kr_{j,\ell})\cdot\operatorname{Im}(u_{j,k})) \\
		=&\mathbf{Var}(|u_{j,k}|^2)+\frac{1}{2}Y_0(kr_{j,\ell})\cdot\mathbf{Cov}(|u_{j,k}|^2,\operatorname{Re}(u_{j,k}))
		-\frac{1}{2}J_0(kr_{j,\ell})\cdot\mathbf{Cov}(|u_{j,k}|^2,\operatorname{Im}(u_{j,k})).
	\end{aligned}
\end{equation}
From \textcolor{blue}{\eqref{v_jlk}} and \textcolor{blue}{\eqref{covuv}}, we obtain
\begin{equation}\label{varv}
	\begin{aligned}
		4\mathbf{Var}(|v_{j,\ell,k}|^2) =& 4\mathbf{Var}(|u_{j,k}|^2+\frac{1}{2}Y_0(kr_{j,\ell})\cdot \operatorname{Re}(u_{j,k}) -\frac{1}{2}J_0(kr_{j,\ell})\cdot\operatorname{Im}(u_{j,k}))\\
		=&4\mathbf{Var}(|u_{j,k}|^2)+Y_0^2(kr_{j,\ell})\cdot\mathbf{Var}(\operatorname{Re}(u_{j,k}))+J_0^2(kr_{j,\ell})\cdot\mathbf{Var}(\operatorname{Im}(u_{j,k}))\\
		&+4Y_0(kr_{j,\ell})\cdot\mathbf{Cov}(|u_{j,k}|^2,\operatorname{Re}(u_{j,k}))-4 J_0(kr_{j,\ell})\cdot\mathbf{Cov}(|u_{j,k}|^2,\operatorname{Im}(u_{j,k}))\\&-2Y_0(kr_{j,\ell})\cdot
		J_0(kr_{j,\ell})\cdot\mathbf{Cov}(\operatorname{Re}(u_{j,k}),\operatorname{Im}(u_{j,k}))\\
		=&Y_0^2(kr_{j,\ell})\cdot\mathbf{Var}(\operatorname{Re}(u_{j,k}))+J_0^2(kr_{j,\ell})\cdot\mathbf{Var}(\operatorname{Im}(u_{j,k}))+8\mathbf{Cov}(|u_{j,k}|^2,|v_{j,\ell,k}|^2)\\&-4\mathbf{Var}(|u_{j,k}|^2)-2Y_0(kr_{j,\ell})\cdot
		J_0(kr_{j,\ell})\cdot\mathbf{Cov}(\operatorname{Re}(u_{j,k}),\operatorname{Im}(u_{j,k}))\\
	\end{aligned}
\end{equation}
and
\begin{equation}\label{covvv}
	\begin{aligned}
		4\mathbf{Cov}(|v_{j,1,k}|^2,|v_{j,2,k}|^2) =&4\mathbf{Cov}(|u_{j,k}|^2+\frac{1}{2}Y_0(kr_{j,1})\cdot \operatorname{Re}(u_{j,k}) -\frac{1}{2}J_0(kr_{j,1})\cdot\operatorname{Im}(u_{j,k}),\\&|u_{j,k}|^2+\frac{1}{2}Y_0(kr_{j,2})\cdot \operatorname{Re}(u_{j,k}) -\frac{1}{2}J_0(kr_{j,2})\cdot\operatorname{Im}(u_{j,k})) \\
		=&4\mathbf{Cov}(|u_{j,k}|^2,|v_{j,1,k}|^2)+4\mathbf{Cov}(|u_{j,k}|^2,|v_{j,2,k}|^2)-4\mathbf{Var}(|u_{j,k}|^2)\\&+Y_0(kr_{j,1})Y_0(kr_{j,2})\cdot\mathbf{Var}(\operatorname{Re}(u_{j,k}))+J_0(kr_{j,1})J_0(kr_{j,2})\cdot\mathbf{Var}(\operatorname{Im}(u_{j,k}))\\&-\left(J_0(kr_{j,1})Y_0(kr_{j,2})+Y_0(kr_{j,1})J_0(kr_{j,2})\right)\cdot\mathbf{Cov}(\operatorname{Re}(u_{j,k}),\operatorname{Im}(u_{j,k})).
	\end{aligned}
\end{equation}
Using equations \textcolor{blue}{\eqref{varv}} and \textcolor{blue}{\eqref{covvv}}, we derive the variance retrieval formula

\begin{equation}\label{Vf}
	\begin{aligned}
		&Y_0^2(kr_{j,1}) \cdot\mathbf{Var}(\operatorname{Re}(u_{j,k})) + J_0^2(kr_{j,1}) \cdot\mathbf{Var}(\operatorname{Im}(u_{j,k})) \\&- 2Y_0(kr_{j,1}) J_0(kr_{j,1}) \cdot\mathbf{Cov}(\operatorname{Re}(u_{j,k}), \operatorname{Im}(u_{j,k})) =F_{j,k,1,1}, \\
		&Y_0^2(kr_{j,2}) \cdot \mathbf{Var}(\operatorname{Re}(u_{j,k})) + J_0^2(kr_{j,2}) \cdot \mathbf{Var}(\operatorname{Im}(u_{j,k})) \\
		&-2Y_0(kr_{j,2}) J_0(kr_{j,2}) \cdot \mathbf{Cov}(\operatorname{Re}(u_{j,k}), \operatorname{Im}(u_{j,k})) = F_{j,k,2,2}, \\
		&Y_0(kr_{j,1}) Y_0(kr_{j,2}) \cdot \mathbf{Var}(\operatorname{Re}(u_{j,k})) + J_0(kr_{j,1}) J_0(kr_{j,2}) \cdot \mathbf{Var}(\operatorname{Im}(u_{j,k})) \\
		&-(Y_0(kr_{j,1}) J_0(kr_{j,2}) + Y_0(kr_{j,2}) J_0(kr_{j,1})) \cdot \mathbf{Cov}(\operatorname{Re}(u_{j,k}), \operatorname{Im}(u_{j,k})) = F_{j,k,1,2},
	\end{aligned}
\end{equation}
where
\begin{equation}\label{F}
	\begin{aligned}
		F_{j,k,\ell_1,\ell_2} =&4[\mathbf{Cov}(|v_{j,\ell_1,k}|^2,|v_{j,\ell_2,k}|^2)-\mathbf{Cov}(|u_{j,k}|^2,|v_{j,\ell_1,k}|^2)-\mathbf{Cov}(|u_{j,k}|^2,|v_{j,\ell_2,k}|^2)\\&+\mathbf{Var}(|u_{j,k}|^2)].
	\end{aligned}
\end{equation}
Thus, by a simple calculation, the variance and covariance of real and imaginary parts on
$\partial B_j$ is given by
\begin{equation}\label{Var}
	\mathbf{Var}(\operatorname{Re}(u_{j,k})) = \frac{\operatorname{det}(D_{j,k}^{(1)})}{\operatorname{det}(D_{j,k})},\quad \mathbf{Var}(\operatorname{Im}(u_{j,k})) = \frac{\operatorname{det}(D_{j,k}^{(2)})}{\operatorname{det}(D_{j,k})},\quad\mathbf{Cov}(\operatorname{Re}(u_{j,k}),\operatorname{Im}(u_{j,k})) = \frac{\operatorname{det}(D_{j,k}^{(3)})}{\operatorname{det}(D_{j,k})},
\end{equation}
where the function matrices $D_{j,k}, D_{j,k}^{(1)}, D_{j,k}^{(2)}, D_{j,k}^{(3)}$ are defined as follows
$$
\begin{aligned}D_{j,k}&=\begin{pmatrix}
		Y_0^2(kr_{j,1})&J_0^2(kr_{j,1})&-2Y_0(kr_{j,1})J_0(kr_{j,1})\\
		Y_0^2(kr_{j,2})&J_0^2(kr_{j,2})&-2Y_0(kr_{j,2})J_0(kr_{j,2})\\
		Y_0(kr_{j,1})Y_0(kr_{j,2})&J_0(kr_{j,1})J_0(kr_{j,2})&-Y_0(kr_{j,1})J_0(kr_{j,2})-Y_0(kr_{j,2})J_0(kr_{j,1})
	\end{pmatrix},\\\\
	D_{j,k}^{(1)}&=\begin{pmatrix}
		F_{j,k,1,1}&J_0^2(kr_{j,1})&-2Y_0(kr_{j,1})J_0(kr_{j,1})\\
		F_{j,k,2,2}&J_0^2(kr_{j,2})&-2Y_0(kr_{j,2})J_0(kr_{j,2})\\
		F_{j,k,1,2}&J_0(kr_{j,1})J_0(kr_{j,2})&-Y_0(kr_{j,1})J_0(kr_{j,2})-Y_0(kr_{j,2})J_0(kr_{j,1})
	\end{pmatrix},\\\\
	D_{j,k}^{(2)}&=\begin{pmatrix}
		Y_0^2(kr_{j,1})&F_{j,k,1,1}&-2Y_0(kr_{j,1})J_0(kr_{j,1})\\
		Y_0^2(kr_{j,2})&F_{j,k,2,2}&-2Y_0(kr_{j,2})J_0(kr_{j,2})\\
		Y_0(kr_{j,1})Y_0(kr_{j,2})&F_{j,k,1,2}&-Y_0(kr_{j,1})J_0(kr_{j,2})-Y_0(kr_{j,2})J_0(kr_{j,1})
	\end{pmatrix},\\\\
	D_{j,k}^{(3)}&=\begin{pmatrix}
		Y_0^2(kr_{j,1})&J_0^2(kr_{j,1})&F_{j,k,1,1}\\
		Y_0^2(kr_{j,2})&J_0^2(kr_{j,2})&F_{j,k,2,2}\\
		Y_0(kr_{j,1})Y_0(kr_{j,2})&J_0(kr_{j,1})J_0(kr_{j,2})&F_{j,k,1,2}
	\end{pmatrix}.
\end{aligned}$$
Therefore, the variance of the radiated fields $\mathbf{Var}(u_{j,k})$ can be recovered from $\mathbf{Var}(u_{j,k}) = \mathbf{Var}(\operatorname{Re}(u_{j,k})) + \mathbf{Var}(\operatorname{Im}(u_{j,k}))$ for
$j = 1, \cdots, m.$
\subsection{Phase retrieval algorithm}
Similar to \textcolor{blue}{\cite{Zhang2018Retri}}, the parameters $R,m$ and $\lambda_{j,\ell}$ for $\ell=1,2$ are taken as:
 \begin{equation}\label{paramater}
	\left.\begin{aligned}
		&m\geqslant10,\quad\lambda_{j,1}=\frac12,\quad k^*=\frac\pi{30a},\quad\tau\geqslant6,\\
		&\left\{\begin{array}{lll}
			R=\tau a,&\lambda_{j,2}=\frac12+\frac\pi{2kR},&\mathrm{~if~}k\in\mathbb{K}_N\backslash\{k^*\},\\
			R=6a,&\lambda_{j,2}=-\frac32,&\mathrm{~if~}k=k^*.
		\end{array}\right.
	\end{aligned}\right.
\end{equation}
The selection of parameters \textcolor{blue}{\eqref{paramater}} ensures the equations  \textcolor{blue}{\eqref{Ef}} and \textcolor{blue}{\eqref{Vf}} are uniquely solvable. It will be discussed in the next section.
However, in most cases, the measured data contain noise, which means that formulas \textcolor{blue}{\eqref{E}} and \textcolor{blue}{\eqref{Var}} cannot be directly applied in computations. Therefore, we present an algorithm with perturbed data. Algorithm \textcolor{blue}{\ref{PR}} describes the PR algorithm process.

\begin{algorithm}[ht]
\caption{Phase retrieval with reference point sources (PR) algorithm}
\label{PR}
\begin{algorithmic}[1]
	\State {Take the parameters $R,m$ and $\lambda_{j,\ell}$ for $j=1,\cdots,m,\ell=1,2$ as in \textcolor{blue}{\eqref{paramater}};}
	\State{Measure the noisy phaseless data on $\partial B_R$ for all $k \in \mathbb{K}_N$: $\{\mathbf{E}^\epsilon(|u(x,k)|^2)\}$ and
		$\{\mathbf{Var}^\epsilon(|u(x,k)|^2)\}$;}
	\State{Introduce the reference point sources $\delta_{j,\ell}$ into the inverse random source system $f$ for $j = 1, \cdots, m$ and $\ell = 1,2$.
		Then, for each $k \in \mathbb{K}_N$, collect the corresponding phaseless statistical data on $\partial B_j$, including:
$\mathbf{E}^\epsilon(|v_{j,1,k}|^2)$, $\mathbf{E}^\epsilon(|v_{j,2,k}|^2)$,
		$\mathbf{Cov}^\epsilon(|v_{j,1,k}|^2, |v_{j,2,k}|^2)$,
		$\mathbf{Cov}^\epsilon(|u_{j,k}|^2, |v_{j,1,k}|^2)$,
		$\mathbf{Cov}^\epsilon(|u_{j,k}|^2,|v_{j,2,k}|^2)$, $\mathbf{Var}^\epsilon(|v_{j,1,k}|^2)$ and $\mathbf{Var}^\epsilon(|v_{j,2,k}|^2)$ for $j=1,\cdots,m$;
	}
	\State{Recover the mean of the radiated fields $\{\mathbf{E}^\epsilon(u(x,k)):x\in\partial B_j,k\in\mathbb{K}_N\}$ for $j=1,\cdots,m$ from formula \textcolor{blue}{\eqref{E}}, the covariance $\mathbf{Cov}^\epsilon(\operatorname{Re}(u_{j,k}), \operatorname{Im}(u_{j,k}))$, and the variance: $\mathbf{Var}^\epsilon(\operatorname{Re}(u_{j,k}))$, $\mathbf{Var}^\epsilon(\operatorname{Im}(u_{j,k}))$ for $j=1,\cdots,m$ from formula \textcolor{blue}{\eqref{Var}}.}
	\end{algorithmic}
\end{algorithm}
\subsection{Stability of phase retrivel}
In this section, we will prove that the norm of the determinant of the coefficient matrix for each system of equations \textcolor{blue}{\eqref{Ef}} and \textcolor{blue}{\eqref{Vf}} has a strictly positive lower bound. This ensures the uniqueness of phase information of the radiated fields. Additionally, we will also further analyze stability results of the recovered statistics.

\begin{thm}\label{det}
	Under \textcolor{blue}{\eqref{paramater}}, we have the following estimate
	
	\begin{equation}\label{detA}
		\left|\det(A_{j,k})\right|\geqslant\begin{cases}\dfrac{M}{kR},\quad\text{if }k\in\mathbb{K}_N\backslash\{k^*\},\\M^*,\quad\text{if }k=k^*,\end{cases}\quad j=1,\cdots,m,
	\end{equation}
	and
	\begin{equation}\label{detD}
		\left|\det(D_{j,k})\right|\geqslant\begin{cases}\left(\dfrac{M}{kR}\right)^3,\quad\text{if }k\in\mathbb{K}_N\backslash\{k^*\},\\(M^*)^3,\quad\text{if }k=k^*,\end{cases}\quad j=1,\cdots,m,
	\end{equation}
	where $M= \frac{20\tau-7}{20\tau }$ and $M^{*}= \frac 49.$
\end{thm}
\begin{proof}
	Regarding \textcolor{blue}{\eqref{detA}}, the proof can be found in \textcolor{blue}{\cite{Zhang2018Retri}}.
	Now, we proof \textcolor{blue}{\eqref{detD}}.
	Set $s_\ell = kr_{j,\ell}$, we have
	\begin{equation}
		\begin{aligned}
			\det(D_{j,k}) =& Y_0^2(s_1)\left[J_0^2(s_2)(-Y_0(s_1)J_0(s_2)-Y_0(s_2)J_0(s_1))+2J_0(s_1)J_0^2(s_2)Y_0(s_2)\right]\\
			&-Y_0^2(s_2)\left[J_0^2(s_1)(-Y_0(s_1)J_0(s_2)-Y_0(s_2)J_0(s_1))+2J_0^2(s_1)J_0(s_2)Y_0(s_1)\right]\\
			&+Y_0(s_1)Y_0(s_2)\left[J_0^2(s_1)(-2Y_0(s_2)J_0(s_2))+2J_0(s_2)^2Y_0(s_1)J_0(s_1)\right]\
			\\
			=&\left(J_0(s_1)\cdot Y_0(s_2)-Y_0(s_1)\cdot J_0(s_2)\right)^3.
		\end{aligned}
	\end{equation}
	Since
	$$\det(A_{j,k}) =J_0(s_1)\cdot Y_0(s_2)-Y_0(s_1)\cdot J_0(s_2),$$
	it follows that
	\begin{equation}
		\det(D_{j,k}) = (\det(A_{j,k}))^3.
	\end{equation}
\end{proof}
\begin{rem}
	We would like to emphasize that, aside from the parameter settings selected in (\textcolor{blue}{\ref{paramater}}), there are alternative strategies available that can ensure the inequalities (\textcolor{blue}{\ref{detA}}) and (\textcolor{blue}{\ref{detD}}) hold, thereby guaranteeing the unique solvability of the two systems of equations (\textcolor{blue}{\ref{Ef}}) and (\textcolor{blue}{\ref{Vf}}).
\end{rem}
Now, we proceed to investigate the stability of the phase retrieval formula. For a fixed $k$ and $j$,
we consider the following perturbed linear system with the unknowns $\mathbf{E}^\epsilon(\mathrm{Re}(u_{j,k}))$ and $\mathbf{E}^\epsilon(\mathrm{Im}(u_{j,k}))$:
\begin{equation}\label{E_eps}
	\begin{aligned}
		&Y_0(kr_{j,1})\cdot\mathbf{E}^\epsilon(\mathrm{Re}(u_{j,k}))-J_0(kr_{j,1})\cdot\mathbf{E}^\epsilon(\mathrm{Im}(u_{j,k}))=f_{j,1,k}^\epsilon,\\
		&Y_0(kr_{j,2})\cdot\mathbf{E}^\epsilon(\mathrm{Re}(u_{j,k}))-J_0(kr_{j,2})\cdot\mathbf{E}^\epsilon(\mathrm{Im}(u_{j,k}))=f_{j,2,k}^\epsilon,
	\end{aligned}
\end{equation}
where

\begin{equation}\label{err_fc}
		f_{j,\ell,k}^{\epsilon} = 2\left(\mathbf{E}^{\epsilon}(|v_{j,\ell,k}|)^{2} - \mathbf{E}^{\epsilon}(|u_{j,k}|)^{2}\right) - \frac{1}{8}\left|H_{0}^{(1)}(kr_{j,\ell})\right|^{2},
	\quad \ell=1,2.
\end{equation}
It is easy to observe that the solutions to the perturbed equations \textcolor{blue}{\eqref{E_eps}} can also be expressed using formula \textcolor{blue}{\eqref{E}}, with $f_{j,\ell,k}^\epsilon$ replacing $f_{j,\ell,k}$ ($\ell = 1, 2$).
Let
\[
\mathcal{E}(u)= \begin{pmatrix} \mathbf{E}(\mathrm{Re}(u_{j,k})) \\ \mathbf{E}(\mathrm{Im}(u_{j,k})) \end{pmatrix}, \quad
\mathcal{F} = \begin{pmatrix} f_{j,1,k} \\ f_{j,2,k} \end{pmatrix},
\]
and
\[
\mathcal{E}^\epsilon(u) = \begin{pmatrix} \mathbf{E}^\epsilon(\mathrm{Re}(u_{j,k})) \\ \mathbf{E}^\epsilon(\mathrm{Im}(u_{j,k})) \end{pmatrix}, \quad
\mathcal{F}^\epsilon = \begin{pmatrix} f_{j,1,k}^\epsilon \\ f_{j,2,k}^\epsilon \end{pmatrix}.
\]
Before analyzing stability, we state a boundedness lemma for $J_0$ and $Y_0$, whose estimates are derived in the proof of Theorem~3.2 of \textcolor{blue}{\cite{Zhang2018Retri}}.

\begin{lem}\label{lemma:J0Y0}
Let $t_\ell=kr_{j,\ell}$ for $\ell=1,2$.

(i) If $k\in\mathbb{K}_N\setminus\{k^*\}$, then
\begin{equation}\label{J0Y01}
    \left|J_0(t_\ell)\right|\leqslant \frac{0.82}{\sqrt{kR(1-\lambda_{j,\ell})}},\qquad
\left|Y_0(t_\ell)\right|\leqslant \frac{0.82}{\sqrt{kR(1-\lambda_{j,\ell})}}.
\end{equation}

(ii) If $k=k^*=\dfrac{\pi}{30a}$, then
\begin{equation}\label{J0Y02}
    \left|J_0(t_\ell)\right|\leqslant 1,\qquad
\left|Y_0(t_\ell)\right|\leqslant 1.
\end{equation}

\end{lem}
We first establish stability for the expectations.
\begin{thm}\label{Esta}
Under \textcolor{blue}{\eqref{paramater}}, it follows that
\begin{align*}
\|\mathcal{E}^\epsilon(u) - \mathcal{E}(u)\|_2
\leqslant C \sum_{\ell=1}^2
\left(\mathbf{E}^\epsilon(|v_{j,\ell,k}|^2) - \mathbf{E}(|v_{j,\ell,k}|^2) \right)
+ \left( \mathbf{E}^\epsilon(|u_{j,k}|^2) - \mathbf{E}(|u_{j,k}|^2)\right),
\end{align*}
where $C>0$ is a constant depends on $k, R, M, M^*, \lambda_{j,2}$.
\end{thm}
\begin{proof}

From \textcolor{blue}{\eqref{Ef}}, the linear systems are given by
\[
A_{j,k} \mathcal{E}(u)= \mathcal{F}, \qquad
A_{j,k} \mathcal{E}^\epsilon(u) = \mathcal{F}^\epsilon.
\]
By inverting \(A_{j,k}\), we obtain
\[
\mathcal{E}(u)= A_{j,k}^{-1} \mathcal{F}, \quad
\mathcal{E}^\epsilon(u) = A_{j,k}^{-1} \mathcal{F}^\epsilon,
\]
which implies
\[
\|\mathcal{E}^\epsilon(u) - \mathcal{E}(u)\|_2
= \|A_{j,k}^{-1} (\mathcal{F}^\epsilon - \mathcal{F})\|_2
\leqslant \|A_{j,k}^{-1}\|_F \, \|\mathcal{F}^\epsilon - \mathcal{F}\|_2.
\]
Using the explicit form
\[
A_{j,k}^{-1} = \frac{1}{|\det A_{j,k}|}
\begin{pmatrix}
- J_0(kr_{j,2}) & J_0(kr_{j,1}) \\[1mm]
- Y_0(kr_{j,2}) &  Y_0(kr_{j,1})
\end{pmatrix}
\]
and applying Lemma~\textcolor{blue}{\ref{lemma:J0Y0}} and \textcolor{blue}{\eqref{detA}}, we deduce that
\[
\|A_{j,k}^{-1}\|_F \leqslant
\begin{cases}
\frac{1.64}{M} \sqrt{\frac{kR}{1-\lambda_{j,2}}}, & k \in \mathbb{K}_N \setminus \{k^*\},\\[1em]
\frac{2}{M^*}, &  k = k^*.
\end{cases}
\]
Finally,  we have
\begin{equation}
\begin{aligned}
    &\| \mathcal{E}^\epsilon(u) - \mathcal{E}(u)\|_2 \leq C_1 \| \mathcal{F}^\epsilon - \mathcal{F} \|_2 \\
    \leqslant& 2C_1(\left( \left( \mathbf{E}^\epsilon (|v_{j,1,k}|^2) - \mathbf{E}^\epsilon (|u_{j,k}|^2) \right) - \left( \mathbf{E}(|v_{j,1,k}|^2) - \mathbf{E}(|u_{j,k}|^2 \right)\right))^2 \\
    &+ \left( \left( \mathbf{E}^\epsilon (|v_{j,2,k}|^2) - \mathbf{E}^\epsilon (|u_{j,k}|^2) \right) - \left( \mathbf{E} (|v_{j,2,k}|^2) - \mathbf{E}(|u_{j,k}|^2)\right) \right)^2 )^{\frac{1}{2}} \\
    =& 2C_1 (\left( \left( \mathbf{E}^\epsilon (|v_{j,1,k}|^2) - \mathbf{E}(|v_{j,1,k}|^2) \right) - \left( \mathbf{E}^\epsilon (|u_{j,k}|^2) - \mathbf{E} (|u_{j,k}|^2) \right) \right)^2 \\
    &+ \left( \left( \mathbf{E}^\epsilon (|v_{j,2,k}|^2) - \mathbf{E} (|v_{j,2,k}|^2) \right) - \left( \mathbf{E}^\epsilon (|u_{j,k}|^2) - \mathbf{E} (|u_{j,k}|^2) \right) \right)^2)^{\frac{1}{2}} \\
    \leqslant& C_2 \left( \left| \mathbf{E}^\epsilon (|v_{j,1,k}|^2) - \mathbf{E}(|v_{j,1,k}|^2) \right| + \left| \mathbf{E}^\epsilon (|v_{j,2,k}|^2) - \mathbf{E}(|v_{j,2,k}|^2) \right| + \left| \mathbf{E}^\epsilon (|u_{j,k}|^2) - \mathbf{E}( |u_{j,k}|^2) \right| \right).
\end{aligned}
\end{equation}
\end{proof}

We now turn to analyze the stability of the variance.
Define
\[
\mathcal{V}(u) =
\begin{pmatrix}
\mathbf{Var}(\mathrm{Re}(u_{j,k})) \\
\mathbf{Var}(\mathrm{Im}(u_{j,k})) \\
\mathbf{Cov}(\mathrm{Re}(u_{j,k}), \mathrm{Im}(u_{j,k}))
\end{pmatrix}, \quad
\mathcal{V}^\epsilon(u) =
\begin{pmatrix}
\mathbf{Var}^\epsilon(\mathrm{Re}(u_{j,k})) \\
\mathbf{Var}^\epsilon(\mathrm{Im}(u_{j,k})) \\
\mathbf{Cov}^\epsilon(\mathrm{Re}(u_{j,k}), \mathrm{Im}(u_{j,k}))
\end{pmatrix},
\]
\begin{equation}
\begin{aligned}
\Delta_{\ell_1,\ell_2}^{vv} &:=
\mathbf{Cov}^\epsilon\big(|v_{j,\ell_1,k}|^2, |v_{j,\ell_2,k}|^2\big)
- \mathbf{Cov}\big(|v_{j,\ell_1,k}|^2, |v_{j,\ell_2,k}|^2\big),\\[1ex]
\Delta_{\ell}^{uv} &:=
\mathbf{Cov}^\epsilon\big(|u_{j,k}|^2, |v_{j,\ell,k}|^2\big)
- \mathbf{Cov}\big(|u_{j,k}|^2, |v_{j,\ell,k}|^2\big),\\[1ex]
\Delta^{uu} &:=
\mathbf{Var}^\epsilon\big(|u_{j,k}|^2\big)
- \mathbf{Var}\big(|u_{j,k}|^2\big).
\end{aligned}\notag
\end{equation}

\begin{thm}\label{var_stab}
Under \textcolor{blue}{\eqref{paramater}}, we have
\begin{align}
\|\mathcal{V}^\epsilon(u) - \mathcal{V}(u)\|_2
\leqslant C \sum_{1\leqslant \ell_1 \leqslant \ell_2 \leqslant 2}
\Delta_{\ell_1,\ell_2}^{vv} + \Delta_{\ell_1}^{uv} + \Delta_{\ell_2}^{uv} + \Delta^{uu},
\end{align}
where $C>0$ is a constant depends on $k, R, M, M^*, \lambda_{j,2}$.
\end{thm}

\begin{proof}
First, consider the real part.

\textbf{Case 1:} $k\in\mathbb{K}_N\backslash\{k^*\}$. From \textcolor{blue}{\eqref{Var}}, \textcolor{blue}{\eqref{detD}} and \textcolor{blue}{\eqref{J0Y01}}, we obtain
\begin{equation}
\begin{aligned}
    &|\mathbf{Var}^\epsilon(\mathrm{Re}(u_{j,k})-\mathbf{Var}(\mathrm{Re}(u_{j,k}))|\\
    \leqslant &\frac{k^3 R^3}{M^3} \bigg( |(F^\epsilon_{j,k,1,1}-F_{j,k,1,1})(-J_0^3(kr_{j,2}) Y_0(k r_{j,1}) - J_0(k r_{j,1}) J_0^2(k r_{j,2}) Y_0(k r_{j,2}) + 2 J_0(k r_{j,1}) J_0^2(k r_{j,2}) Y_0(k r_{j,2}))| \\
    & +| (F^\epsilon_{j,k,2,2}-F_{j,k,2,2})(J_0^2(k r_{j,1}) J_0(k r_{j,2}) Y_0(k r_{j,1}) +J_0^3(k r_{j,1}) Y_0(k r_{j,2}) - 2 J_0^2(k r_{j,1}) J_0(k r_{j,2}) Y_0(k r_{j,1})) | \\
    &+ | (F^\epsilon_{j,k,1,2}-F_{j,k,1,2})( -2 J_0^2(k r_{j,1}) J_0(k r_{j,2}) Y_0(k r_{j,2}) + 2 J_0(k r_{j,1}) J_0^2(k r_{j,2}) Y_0(k r_{j,1})) | \bigg) \\
    \leqslant& \frac{16k^3 R^3}{M^3} \left( \frac{0.82}{\sqrt{kR(1-\lambda_{j,2})}} \right)^4 \sum_{1\leqslant l_1\leqslant l_2\leqslant 2} ( |\mathbf{Cov}^\epsilon(|v_{j,l_1,k}|^2, |v_{j,l_2,k}|^2)-\mathbf{Cov}(|v_{j,l_1,k}|^2, |v_{j,l_2,k}|^2)|\\
    &\quad +|\mathbf{Cov}^\epsilon(|u_{j,k}|^2, |v_{j,l_1,k}|^2)-\mathbf{Cov}(|u_{j,k}|^2, |v_{j,l_1,k}|^2)| +|\mathbf{Cov}^\epsilon(|u_{j,k}|^2, |v_{j,l_2,k}|^2)-\mathbf{Cov}(|u_{j,k}|^2, |v_{j,l_2,k}|^2)| \\
    &\quad + |\mathbf{Var}^\epsilon(|u_{j,k}|^2)- \mathbf{Var}(|u_{j,k}|^2)| ) \\
    \leqslant&C_1\sum_{1\leqslant l_1\leqslant l_2\leqslant 2}\Delta_{\ell_1,\ell_2}^{vv} +\Delta_{\ell_1}^{uv}+\Delta_{\ell_2}^{uv}+\Delta^{uu}
\end{aligned}\notag
\end{equation}
with
$$
C_1 = \frac{16 k^3 R^3}{M^3} \left( \frac{0.82}{\sqrt{kR(1-\lambda_{j,2})}} \right)^4.
$$

\textbf{Case 2:} $k=k^*$. A uniform bound holds:
\begin{equation}
|\mathbf{Var}^\epsilon(\mathrm{Re}(u_{j,k^*})) - \mathbf{Var}(\mathrm{Re}(u_{j,k^*}))|
\leqslant C_2 \sum_{1\leqslant \ell_1 \leqslant \ell_2 \leqslant 2}
\Delta_{\ell_1,\ell_2}^{vv} + \Delta_{\ell_1}^{uv} + \Delta_{\ell_2}^{uv} + \Delta^{uu},
\end{equation}
with
$$
C_2 = \frac{16}{M^*}.
$$
Analogous estimates hold for the imaginary part and for the covariance term.
Hence, we obtain the \(2\)-norm bound:
\begin{equation}
\|\mathcal{V}^\epsilon(u) - \mathcal{V}(u)\|_2
\leqslant C \sum_{1\leqslant \ell_1 \leqslant \ell_2 \leqslant 2}
\Delta_{\ell_1,\ell_2}^{vv} + \Delta_{\ell_1}^{uv} + \Delta_{\ell_2}^{uv} + \Delta^{uu},
\end{equation}
where
\[
C = \begin{cases}
\sqrt 3 C_1, & k \in \mathbb{K}_N \setminus \{k^*\},\\
\sqrt 3 C_2, & k = k^*.
\end{cases}
\]
This completes the proof.
\end{proof}

\section{Stochastic inverse problem}\label{3}
In this section, we derive the Fredholm integral equations to reconstruct random source by examining the expectation and variance of the mild solution. Furthermore, the stability of the IRSP has been proved, and thus the stability of the phaseless IRSP is also established.
\subsection{Integral equations}

By taking the expectation on both sides of \textcolor{blue}{\eqref{solve}} and applying
\begin{equation}\label{Ewy}
	\mathbf{E}\left(\int_{\Omega_0}G_k(x,y)\sigma(y)\mathrm{d}{W}_y\right) = 0,
\end{equation}
it deduces
\begin{equation}\label{Eu}
	\mathbf{E}(u(x,k))=\int_{\Omega_0} G_k(x,y)g(y)\mathrm{d}y.
\end{equation}
This equation can be employed to reconstruct $g$. The reconstruction formula for $g$ is similar to a deterministic inverse problem. The difference is that the known boundary data is obtained from the expectation of the radiation wave field. To simplify the solution process, complex-valued quantities are separated into their real and imaginary parts. More explicitly, we have
\begin{equation}\label{reu}
	\mathrm{Re}u(x,k)=\int_{\Omega_0}\mathrm{Re}G_k(x,y)g(y)\mathrm{d}y+\int_{\Omega_0}\mathrm{Re}G_k(x,y)\sigma(y)\mathrm{d}{W}_{y}\end{equation}
and
\begin{equation}\label{imu}
	\mathrm{Im}u(x,k)=\int_{\Omega_0}\mathrm{Im}G_k(x,y)g(y)\mathrm{d}y+\int_{\Omega_0}\mathrm{Im}G_k(x,y)\sigma(y)\mathrm{d}{W}_{y}.
\end{equation}

Taking the expectation on both sides of \textcolor{blue}{\eqref{reu}} and \textcolor{blue}{\eqref{imu}}, we obtain
\begin{equation}\label{Ereu}
	\mathbf{E}(\mathrm{Re}u(x,k))=\int_{\Omega_0}\mathrm{Re}G_k(x,y)g(y)\mathrm{d}y,
\end{equation}
\begin{equation}\label{Eimu}
	\mathbf{E}(\mathrm{Im}u(x,k))=\int_{\Omega_0}\mathrm{Im}G_k(x,y)g(y)\mathrm{d}y.\end{equation}
Substituting the real and imaginary parts of the two-dimensional Green function into the above equations yields
\begin{equation}\label{Ereu2}
	\mathbf{E}(\mathrm{Re}u(x,k))=\frac14\int_{\Omega_0}Y_0(k|x-y|)g(y)\mathrm{d}y,\end{equation}
\begin{equation}\label{Eimu2}
	\mathbf{E}(\mathrm{Im}u(x,k))=-\frac14\int_{\Omega_0}J_0(k|x-y|)g(y)\mathrm{d}y.
\end{equation}
We can combine two equations to perform the inversion of $g$.

From Lemma \textcolor{blue}{\ref{lem_Wx}} of Appendix \textcolor{blue}{\ref{pf}}, it finds
\begin{equation}\label{ERe2}
	\mathbf{E}\left(\left|\int_{\Omega_0}\mathrm{Re}G_k(x,y)\sigma(y)\mathrm{d}{W}_y\right|^2\right)=\int_{\Omega_0}|\mathrm{Re}G_k(x,y)|^2\sigma^2(y)\mathrm{d}y,
\end{equation}
\begin{equation}\label{Eimu2}
	\mathbf{E}\left(\left|\int_{\Omega_0}\mathrm{Im}G_k(x,y)\sigma(y)\mathrm{d}{W}_y\right|^2\right)=\int_{\Omega_0}|\mathrm{Im}G_k(x,y)|^2\sigma^2(y)\mathrm{d}y
\end{equation}
and
\begin{equation}\label{Ereim}
	\mathbf{E}\left(\int_{\Omega_0}\mathrm{Re}G_k(x,y)\sigma(y)\mathrm{d}{W}_y\cdot \int_{\Omega_0}\mathrm{Im}G_k(x,y)\sigma(y)\mathrm{d}{W}_y\right)=\int_{\Omega_0}\mathrm{Re}G_k(x,y)\mathrm{Im}G_k(x,y)\sigma^2(y)\mathrm{d}y.
\end{equation}
By taking the variance on both sides of \textcolor{blue}{\eqref{reu}} and \textcolor{blue}{\eqref{imu}}, it leads to the following equations

\begin{equation}\label{varre1}
	\mathbf{Var}(\mathrm{Re}u(x,k))=\int_{\Omega_0}|\mathrm{Re}G_k(x,y)|^2\sigma^2(y)\mathrm{d}y,
\end{equation}

\begin{equation}\label{varim1}
	\mathbf{Var}(\mathrm{Im}u(x,k))=\int_{\Omega_0}|\mathrm{Im}G_k(x,y)|^2\sigma^2(y)\mathrm{d}y.
\end{equation}
Furthermore, substituting the real and imaginary parts of Green function into \textcolor{blue}{\eqref{varre1}} and \textcolor{blue}{\eqref{varim1}}, we obtain
\begin{equation}\label{varre}
	\mathbf{Var}(\mathrm{Re}u(x,k))=\frac1{16}\int_{\Omega_0}Y_0^2(k|x-y|)\sigma^2(y)\mathrm{d}y,\end{equation}
\begin{equation}\label{varim}
	\mathbf{Var}(\mathrm{Im}u(x,k))=\frac1{16}\int_{\Omega_0}J_0^2(k|x-y|)\sigma^2(y)\mathrm{d}y.\end{equation}
Then we can recover the variance from above Fredholm integral equations.
Using \textcolor{blue}{\eqref{Ewy}}, \textcolor{blue}{\eqref{reu}}-\textcolor{blue}{\eqref{Eimu}} and \textcolor{blue}{\eqref{Ereim}}, it follows that
\begin{equation}
	\begin{aligned}
		\mathbf{E}(\mathrm{Re}u(x,k)\cdot\mathrm{Im}u(x,k))
		=&\int_{\Omega_0}\mathrm{Re}G_k(x,y)g(y)\mathrm{d}y\cdot\int_{\Omega_0}\mathrm{Im}G_k(x,y)g(y)\mathrm{d}y\\&+\mathbf{E}(\int_{\Omega_0}\mathrm{Re}G_k(x,y)\sigma(y)\mathrm{d}{W}_{y}\cdot\int_{\Omega_0}\mathrm{Im}G_k(x,y)\sigma(y)\mathrm{d}{W}_{y})\\
		=&\mathbf{E}(\mathrm{Re}u(x,k))\cdot\mathbf{E}(\mathrm{Im}u(x,k))+\int_{\Omega_0}\mathrm{Re}G_k(x,y)\mathrm{Im}G_k(x,y)\sigma^2(y)\mathrm{d}y.
	\end{aligned}
\end{equation}
Hence, we obtain
\begin{equation}\label{covreim1}
	\mathbf{Cov}(\mathrm{Re}u(x,k),\mathrm{Im}u(x,k)) = \int_{\Omega_0}\mathrm{Re}G_k(x,y)\mathrm{Im}G_k(x,y)\sigma^2(y)\mathrm{d}y,
\end{equation}
that is,
\begin{equation}\label{covreim}
	\mathbf{Cov}(\mathrm{Re}u(x,k),\mathrm{Im}u(x,k)) = -\frac{1}{16}\int_{\Omega_0}Y_0(k|x-y|)J_0(k|x-y|)\sigma^2(y)\mathrm{d}y.
\end{equation}
\subsubsection{Stability of integral equations}
In this section, we establish the stability result of the integral equations \textcolor{blue}{\eqref{Eu}} and \textcolor{blue}{\eqref{varim1}}.
\begin{thm}\label{sta_g}
	Assume that $g\in L^2(\Omega_0)$ and $\|g\|_{H_p(\Omega_0)}\leqslant M$. Then, for any \( k \in \mathbb{K}_N  \), the following estimate holds
\begin{equation}\label{sta_g1}
\|g(x)\|_{L^2(\Omega_0)}\leqslant C_{k,p,M}\| \mathbf{E} (u(x, k))\|^\frac{p}{p+2}_{L^2(\mathbb{R}^2)},
\end{equation}
	where $$C_{k,p,M}=(2\pi)^\frac{p}{2(p+2)}k^\frac{2p}{p+2}M^\frac{2}{p+2}.$$
\end{thm}

\begin{proof}
	Extend $g(x)$ by zero from $\Omega_0$ to $\mathbb{R}^2$, defining it as follows:
	\[
	g(x) =
	\begin{cases}
		g(x), & x \in \Omega_0, \\
		0, & x \in \mathbb{R}^2 \setminus \Omega_0.
	\end{cases}
	\]
	Setting
	\[
	Eu(x):= \mathbf{E}(u(x,k)).
	\]
	From (\textcolor{blue}{\ref{Eu}}), it follows that
	\[
	Eu(x) = \int_{\Omega_0} G_k(x,y) g(y) \mathrm{d}y = \int_{\mathbb{R}^2} G_k(x,y) g(y) \mathrm{d}y = (G_k * g)(x),
	\]
	where \(G_k(x,y)\) is the Green's function and \(*\) denotes convolution. Applying the Fourier transform to both sides, and using the convolution theorem, we obtain
	\[
	\widehat{Eu}(\omega) = \widehat{G_k}(\omega) \cdot \hat{g}(\omega),
	\]
	where $\widehat{G_k}(\omega)$ is the Fourier transform of $G_k(x)$, and $\hat{g}(\omega)$ is the Fourier transform of $g(x)$.
	
	Using Parseval's theorem, H\"{o}lder inequality and $\widehat{G_k}(\omega) = \frac{1}{\sqrt{2\pi}}\frac{1}{k^2-|\omega|^2}$, we have
	\[
	\begin{aligned}
		\|g(x)\|_{L^2(\mathbb{R}^2)}^2 = \|\hat{g}(\omega)\|_{L^2(\mathbb{R}^2)}^2
		&= \int_{\mathbb{R}^2} \left| \frac{1}{\widehat{G_k}(\omega)} \right|^2 \left| \widehat{Eu}(\omega) \right|^{\frac{4}{p+2}} \left| \widehat{Eu}(\omega) \right|^{\frac{2p}{p+2}} \mathrm{d}\omega \\
		&\leqslant \left[ \int_{\mathbb{R}^2} \left( \frac{ \left| \widehat{Eu}(\omega) \right| }{ \left| \widehat{G_k}(\omega) \right|^2 }^{\frac{4}{p+2}} \right)^{\frac{p+2}{2}} \mathrm{d}\omega \right]^{\frac{2}{p+2}} \cdot \left[ \int_{\mathbb{R}^2} \left( \left| \widehat{Eu}(\omega) \right|^{\frac{2p}{p+2}} \right)^{\frac{p+2}{p}} \mathrm{d}\omega \right]^{\frac{p}{p+2}} \\
		&= \left( \int_{\mathbb{R}^2} (2\pi)^{\frac{p}{2}} \left( \frac{k^2 - |\omega|^2}{1+|\omega|^2} \right)^p |\hat{g}(\omega)|^2 (1+|\omega|^2)^p \mathrm{d}\omega \right)^{\frac{2}{p+2}} \cdot \| \widehat{Eu} \|_{L^2(\mathbb{R}^2)}^{\frac{2p}{p+2}} \\
		&\leqslant (2\pi)^{\frac{p}{p+2}} k^{\frac{4p}{p+2}} \|g\|_{H^p(\mathbb{R}^2)}^{\frac{4}{p+2}} \cdot \| Eu \|_{L^2(\mathbb{R}^2)}^{\frac{2p}{p+2}}.
	\end{aligned}
	\]
	Combining the above estimates, we conclude that
	\[
	\|g(x)\|_{L^2(\Omega_0)} \leqslant C_k \|\mathbf{E} (u(x, k))\|^\frac{p}{p+2}_{L^2(\mathbb{R}^2)},
	\]
	where
	\[
	C_k = (2\pi)^{\frac{p}{2(p+2)}} k^{\frac{2p}{p+2}} M^{\frac{2}{p+2}}.
	\]
\end{proof}

It shows that the norm of \( g(x) \) can be controlled by the norm of \( \mathbf{E}(u(x, k)) \) in \( L^2(\mathbb{R}^2) \). We now proceed to prove that $\sigma^2(x)$ can similarly be controlled.
\begin{lem}\label{low_control_big}
	Let \( f(x) \in L^2(\mathbb{R}^2) \) with \( \|f(x)\|^2_{L^2(\mathbb{R}^2)} = A > 0 \). Then, for any \( \epsilon > 0 \), there exists \( N'(\epsilon) > 0 \) such that for all \( N \geqslant N' \),
	\begin{equation*}
		\int_{|\omega|\geqslant N} |\hat{f}(\omega)|^2 \,\mathrm{d}\omega < \frac{A}{2} < \int_{|\omega|\leqslant N} |\hat{f}(\omega)|^2 \,\mathrm{d}\omega.
	\end{equation*}
\end{lem}

\begin{proof}
	By Parseval's theorem, we have
	\begin{equation*}
		A = \|f(x)\|^2_{L^2(\mathbb{R}^2)} = \|\hat{f}(\omega)\|^2_{L^2(\mathbb{R}^2)} = \int_{\mathbb{R}^2} |\hat{f}(\omega)|^2 \,\mathrm{d}\omega.
		\notag
	\end{equation*}
	Taking the limit as \( N \to \infty \), we obtain
	\begin{equation*}
		\lim_{N \to \infty} \int_{|\omega|\geqslant N} |\hat{f}(\omega)|^2 \,\mathrm{d}\omega = 0.
		\notag
	\end{equation*}
	Thus, for any \( \epsilon > 0 \), there exists \( N_1(\epsilon) > 0 \) such that for all \( N > N_1 \),
	\notag
	\begin{equation*}
		\int_{|\omega|\geqslant N} |\hat{f}(\omega)|^2 \,\mathrm{d}\omega < \epsilon.
	\end{equation*}
	Moreover, we also have
	\begin{equation*}
		\lim_{N \to \infty} \int_{|\omega|\leqslant N} |\hat{f}(\omega)|^2 \,\mathrm{d}\omega = A.
		\notag
	\end{equation*}
	Hence, for any \( \epsilon > 0 \), there exists \( N_2(\epsilon) > 0 \) such that for all \( N > N_2 \),
	\begin{equation*}
		\left| \int_{|\omega|\leqslant N} |\hat{f}(\omega)|^2 \,\mathrm{d}\omega - A \right| < \epsilon,
		\notag
	\end{equation*}
	which implies
	\begin{equation*}
		A - \epsilon < \int_{|\omega|\leqslant N} |\hat{f}(\omega)|^2 \,\mathrm{d}\omega < A + \epsilon.
		\notag
	\end{equation*}
	Setting \( \epsilon = \frac{A}{2} \) and choosing \( N_3 = \max\{N_1, N_2\} \), we obtain that for all \( N > N_3 \),
	\begin{equation*}
		\int_{|\omega|\geqslant N} |\hat{f}(\omega)|^2 \,\mathrm{d}\omega < \frac{A}{2} < \int_{|\omega|\leqslant N} |\hat{f}(\omega)|^2 \,\mathrm{d}\omega.
		\notag
	\end{equation*}
\end{proof}
\begin{rem}
	It demonstrate that the energy of \( f(x) \) is concentrated in the low-frequency region, meaning that the low-frequency energy can effectively control the high-frequency energy.
\end{rem}

\begin{thm}\label{sta_s}
	Assume that $\sigma^2\in L^2(\Omega_0)$. Then, for any \( k \in \mathbb{K}_N  \), the following estimate holds
\begin{equation}\label{sta_s1}
	\|\sigma^2(x)\|_{L^2(\Omega_0)}\leqslant \tilde{C}_k\|\mathbf{Var}(\operatorname{Im}(u(x,k))\|_{L^2(\mathbb{R}^2)},
\end{equation}
	where
	$$\tilde{C}_k = 8\sqrt{2}\pi k^2.$$
\end{thm}
\begin{proof}
	Extend $\sigma^2(x)$ by zero from $\Omega_0$ to $\mathbb{R}^2$ as
	\[
	\sigma^2(x) =
	\begin{cases}
		\sigma^2(x), & x \in \Omega_0, \\
		0, & x \in \mathbb{R}^2 \setminus \Omega_0.
	\end{cases}
	\]
	From (\textcolor{blue}{\ref{varim}}), it follows that
	\[
	\mathbf{Var}(\operatorname{Im}(u(x,k))) = \frac{1}{16} \int_{\Omega_0} J_0^2(k|x-y|) \sigma^2(y) \, \mathrm{d}y = \frac{1}{16} \left(J_0^2(k|\cdot|) * \sigma^2 \right)(x).
	\]
	Taking Fourier transform on both sides and using the convolution theorem, we get
	\[
	\widehat{\mathbf{Var}(\operatorname{Im}(u(\cdot,k)))}(\omega) = \frac{1}{16} \widehat{J_0^2}(k|\omega|) \cdot \widehat{\sigma^2}(\omega).
	\]
	Recall that
	\[
	\widehat{J_0^2}(k|\omega|) =
	\begin{cases}
		\frac{2}{\pi} \frac{1}{|\omega| \sqrt{4k^2 - |\omega|^2}}, & 0 < |\omega| < 2k, \\
		0, & |\omega| \geqslant 2k.
	\end{cases}
	\]
	Using Parseval’s theorem and Lemma \textcolor{blue}{\ref{low_control_big}}, we have
	\[
	\begin{aligned}
		\|\mathbf{Var}(\operatorname{Im}(u(\cdot,k)))\|^2_{L^2(\mathbb{R}^2)} &= \|\widehat{\mathbf{Var}(\operatorname{Im}(u(\cdot,k)))}\|^2_{L^2(\mathbb{R}^2)} \\
		&= \frac{1}{256} \int_{\mathbb{R}^2} \left|\widehat{J_0^2}(k|\omega|)\right|^2 \left|\widehat{\sigma^2}(\omega)\right|^2 \mathrm{d}\omega \\
		&= \frac{1}{256} \int_{|\omega| < 2k} \frac{4}{\pi^2} \frac{1}{|\omega|^2 (4k^2 - |\omega|^2)} \left|\widehat{\sigma^2}(\omega)\right|^2 \mathrm{d}\omega \\
		&\geqslant \frac{1}{128 \pi^2 k^4} \int_{\mathbb{R}^2} \left|\widehat{\sigma^2}(\omega)\right|^2 \mathrm{d}\omega.
	\end{aligned}
	\]
	By Parseval's identity, this implies
	\[
	\|\sigma^2(x)\|_{L^2(\Omega_0)} \leqslant C'_k \|\mathbf{Var}(\operatorname{Im}(u(\cdot,k)))\|_{L^2(\mathbb{R}^2)},
	\]
	where
	\[
	\tilde{C}_k = 8 \sqrt{2} \pi k^2.
	\]
\end{proof}
\begin{rem}
	It shows that the norm of \( \sigma^2(x) \) can be controlled by the norm of \( \mathbf{Var}(\operatorname{Im}(u(x,k)) \) in \( L^2(\mathbb{R}^2) \). However, the same conclusion does not directly apply to $\mathbf{Var}(\operatorname{Re}(u(x,k))$ and $\mathbf{Cov}(\operatorname{Re}(u(x,k),\operatorname{Im}(u(x,k))$. In addition, it should be noted that the right-hand sides of stability estimates (\textcolor{blue}{\ref{sta_g1}}) and (\textcolor{blue}{\ref{sta_s1}}) are based on statistical data of phase-scattered fields in the whole space, rather than data from measurement curve (circle). Therefore, we cannot derive the stability for phaseless IRSP from the stability estimates of the phase retrieval formulas. Unless the scattered field exhibits exceptional regularity, such as analyticity, we cannot control the whole space data using the data from low-dimensional measurement curve. However, the scattered fields induced by random sources often lack such regularity. Thus, we need to develop new methods to obtain stability estimates for IRSP with phaseless statistical data, which will be an important focus for us in the future.
\end{rem}

\section{Bayesian inference}\label{4}
Bayesian inference has become a powerful framework for solving inverse problems, as it allows for the quantification of uncertainty in the inferred parameters \textcolor{blue}{\cite{Wang2004, Kaipio2005}}. Consider the inverse problem of finding $\phi$ from $\psi$, where the relationship between $\phi$ and $\psi$ is given by the model:
\begin{equation}\label{bayesian model}
	\psi=\mathcal{T}(\phi)+\eta,
\end{equation}
here, $\mathcal{T}: X \to \mathbb{R}^{m}$ is a measurable mapping, referred to as the forward operator. $X$ is a separable Hilbert space with inner product $\langle\cdot ,\cdot \rangle_{ X }$. $\phi\in X$ is the unknown function. The observed data $\psi\in\mathbb{R}^{m}$ is finite-dimensional. The noise $\eta$ follows an $m$-dimensional Gaussian noise with zero mean and covariance operator as $\Sigma$, i.e., $\eta\thicksim N(0,\Sigma)$.

The Bayesian approach relies on deriving the posterior distribution of the unknown parameter
$\phi$ based on a prior distribution and the observed data, with the forward model connecting the parameter space to the data space. Let $\mu_{pr}$ be the prior measure on $\phi$. The posterior distribution of $\phi$, conditioned on the observed data $\psi$, is given by the measure $\mu^\psi$, which satisfies the following relation:

\begin{equation}\label{HybridPriorMeasure}
	\frac{d\mu^{\psi}}{d\mu_{pr}}(\phi) \propto\exp(-\Psi(\phi)),
\end{equation}
where the left hand side of \textcolor{blue}{\eqref{HybridPriorMeasure}} is the Radon-Nikodym derivative of the posterior
distribution $\mu^\psi$ with respect to the prior $\mu_{pr}$ and
\begin{equation}\label{potential}
	\Psi(\phi):=\frac12\big\|\mathcal{T}(\phi)-\psi\big\|^2_\Sigma=\frac12\big\|\Sigma^{-1/2}(\mathcal{T}(\phi)-\psi)\big\|_2^2,
\end{equation}
which is commonly referred to as the data fidelity term in deterministic inverse problems.

It is well known that Markov Chain Monte Carlo (MCMC) methods are widely used in Bayesian inference to explore the posterior distribution. In our work, we implement the Preconditioned Crank-Nicolson (pCN) algorithm to generate samples from the posterior distribution $\mu^\Psi$ given by equation \textcolor{blue}{\eqref{HybridPriorMeasure}}. This method, originally developed in \textcolor{blue}{\cite{Stuart2015, Stuart2013, Stuart2010}}, is particularly advantageous due to its dimension-independent properties.

\subsection{Well-posedness}\label{well-poseness}
In this section, we discuss the well-posedness of the posterior distribution arising from the Bayesian inference with Gaussian prior. The regularity properties of Green's function play an important role in the subsequent analysis.

\begin{lem}\label{G_Lp}
	Let $\Omega_0\subset \mathbb{R}^2$ be a bounded domain. It holds that for any $p>0$, $G_k(x,y)\in L^p(\Omega_0)$ for any $y\in \Omega_0$.
\end{lem}
\begin{proof}
	Set $\rho = \sup_{x,y \in \Omega_0} |x - y|$, it implies $\bar{\Omega}_0 \subset B_{\rho}(y)$. Notice that
	\[
	G_k(x, y) = -\frac{1}{2\pi} \log \frac{1}{|x - y|} + V(x, y),
	\]
	here $V$ is a continuous function.
	
	For any $p>0$, it follows that
	\[
	\int_{\Omega_0} \left| \log \frac{1}{|x - y|} \right|^p \, dx \leq \int_{B_{\rho}(y)} \left| \log \frac{1}{|x - y|} \right|^p \, dx \leqslant C \int_0^{\rho} r \left| \log \frac{1}{r} \right|^p\, dr < \infty.
	\]
	This completes the proof.
\end{proof}

We now derive the forward operator equation.
By combining equations \textcolor{blue}{(\ref{Ef})}, \textcolor{blue}{(\ref{Ereu})}, and \textcolor{blue}{(\ref{Eimu})}, we can obtain
\begin{equation}\label{exact_forward_f}
	T_{\ell,k}(g)=f_{\ell,k},\ \ell=1,2,
\end{equation}
where
$$T_{\ell,k}(g) = (T_{1,\ell,k}(g), T_{2,\ell,k}(g),\cdots,T_{m,\ell,k}(g))^\top,f_{\ell,k} = (f_{1,\ell,k},f_{2,\ell,k},\cdots,f_{m,\ell,k})^\top,$$
with
$$T_{j,\ell,k}(g) = Y_0(kr_{j,\ell})\int_{\Omega_0}\mathrm{Re}G_k(x,y)g(y)\mathrm{d}y-J_0(kr_{j,\ell})\int_{\Omega_0}\mathrm{Im}G_k(x,y)g(y)\mathrm{d}y$$
and
$$f_{j,\ell,k} = 2\left(\mathbf{E}(|v_{j,\ell,k}|^{2})-\mathbf{E}(|u_{j,k}|^{2})\right)-\frac{1}{8}\left|H_{0}^{(1)}(kr_{j,\ell})\right|^{2}.$$
Hence the noisy measured data $f_{\ell,k}$ are generated by
\begin{equation}\label{err_forward_f}
	f^\epsilon_{\ell,k} = T_{\ell,k}(g) + \eta_{\ell,k},\ \ell=1,2,
\end{equation}
where $\eta_{\ell,k} \thicksim \mathcal{N}(0,\Sigma_{\ell,k})$ .

The forward operator $T_{\ell,k}:L^{2}(\Omega_0)\rightarrow \mathbb{R}^m$ for the inversion of $g$ needs to satisfy the following lemma as in \textcolor{blue}{\cite{Stuart2010}}:
\begin{lem}\label{forward operator assume1} \ \\
	(i) For every $\delta > 0$ there exists $M_{\ell,k}=M_{\ell,k}(\delta) \in \mathbb{R}$ such that, for all $g \in L^{2}(\Omega_0)$,
	$$\|T_{\ell,k}(g) \|_{\Sigma_{\ell,k}}\leqslant \exp (\delta \|g\|^{2}_{L^2(\Omega_0)}+M_{\ell,k}),$$
	(ii) For every $ r>0$, there exists $K_{\ell,k}=K_{\ell,k}(r)>0$ such that, for all $g_1,g_2 \in L^2(\Omega_0)$ with $$\max\left\{\|g_1\|_{L^2(\Omega_0)},\|g_2\|_{L^2(\Omega_0)}\right\}<r,$$
	we have$$\big\|T_{\ell,k}(g_1) -T_{\ell,k}(g_2) \big\|_{\Sigma_{\ell,k}}\leqslant K_{\ell,k}\|g_1-g_2\|_{L^2(\Omega_0)}.$$
\end{lem}
\begin{proof}
	
	Throughout the proof, the constant $C_{\ell,k}$ changes from occurrence to occurrence.
	For any $g\in L^2(\Omega_0)$, according to Triangle inequality, Cauchy-Schwarz inequality, Lemma \textcolor{blue}{\ref{lemma:J0Y0}} and Lemma \textcolor{blue}{\ref{G_Lp}}, we estimate
	\begin{equation}\label{T1}
		\begin{aligned}
			&\|T_{\ell,k}(g)\|_{\Sigma_{\ell,k}}\\\leqslant& C_{\ell,k}\sum_{j=1}^m \left|Y_0(kr_{j,\ell})\int_{\Omega_0}\mathrm{Re}G_k(x,y)g(y)\mathrm{d}y-J_0(kr_{j,\ell})\int_{\Omega_0}\mathrm{Im}G_k(x,y)g(y)\mathrm{d}y\right|
			\\\leqslant &C_{\ell,k}\left[\left(\int_{\Omega_0}\left|\mathrm{Re}G_k(x,y)\right|^2\mathrm{d}y\right)^{\frac{1}{2}}\left(\int_{\Omega_0}\left(g(y)\right)^2\mathrm{d}y\right)^{\frac{1}{2}}+\left(\int_{\Omega_0}\left|\mathrm{Im}G_k(x,y)\right|^2\mathrm{d}y\right)^{\frac{1}{2}}\left(\int_{\Omega_0}\left(g(y)\right)^2\mathrm{d}y\right)^{\frac{1}{2}}\right]
			\\\leqslant&C_{\ell,k}\|g(y)\|_{L^2(\Omega_0)}
			\\\leqslant&C_{\ell,k}\exp(\|g(y)\|_{L^2(\Omega_0)}).
		\end{aligned}
	\end{equation}
	From Young's inequality, for every $\delta>0$, it follows that
	\begin{equation}\label{T2}
		\begin{aligned}
			\exp(\|g(y)\|_{L^2(\Omega_0)})\leqslant\exp\left(\frac{1}{2\delta}+\frac{\delta \|g(y)\|^2_{L^2(\Omega_0)}}{2}\right).
	\end{aligned}\end{equation}
	The combination of \textcolor{blue}{\eqref{T1}} and \textcolor{blue}{\eqref{T2}} yield
	$$\|T_{\ell,k}(g) \|_{\Sigma_{\ell,k}}\leqslant \exp (\delta \|g\|^{2}_{L^2(\Omega_0)}+M_{\ell,k}).$$
	Using the linearity of $T_{\ell,k}$, we can easily obtain
	$$\|T_{\ell,k}(g_1)-T_{\ell,k}(g_2)\|_{\Sigma_{\ell,k}}\leqslant K_{\ell,k}\|g_1(y)-g_2(y)\|_{L^2(\Omega_0)}.$$
	This completes the proof.
\end{proof}

\begin{defn}[\textcolor{blue}{\cite{Stuart2010}}]
	Assume that two probability measures $\mu$ and $\mu'$ both absolutely continuous with respect to the same reference measure $\nu$, then Hellinger metric with respect to measure $\mu$ and $\mu'$ is defined by
	$$d_{Hell}(\mu, \mu')=\sqrt{\frac12\int_{\Omega_0}\big(\sqrt{\frac{d\mu}{d\nu}}-\sqrt{\frac{d\mu'}{d\nu}}\big)^2 d\nu}.$$
\end{defn}

The Lemma \textcolor{blue}{\ref{forward operator assume1}} about $T_{\ell,k}$ can derive the bounds and Lipschitz properties of the corresponding data fidelity term $\Psi_{\ell,k}$ as Assumptions 2.6 in \textcolor{blue}{\cite{Stuart2010}}.
As a result, the probability measure $\mu^{f^\epsilon_{\ell,k}}$ defined by equation \textcolor{blue}{\eqref{HybridPriorMeasure}} is well defined on $L^2(\Omega_0)$ and it is Lipschitz continuous with respect to the data $f^\epsilon_{\ell,k}$ under the Hellinger metric, as stated in the following theorem.

\begin{thm}\label{welldefined1}
	Let forward operator $T_{\ell,k}:L^2(\Omega_0)\rightarrow \mathbb{R}^m$ satisfies Lemma \ref{forward operator assume1}. For a given $f^\epsilon_{\ell,k}\in \mathbb{R}^m$,  $\mu^{f^\epsilon_{\ell,k}}$ is given by equation (\ref{HybridPriorMeasure}). Then we have the following:\\
	$(i)$ $\mu^{f^\epsilon_{\ell,k}}$ defined as equation (\ref{HybridPriorMeasure}) is well defined on $L^2(\Omega_0)$.\\
	$(ii)$ $\mu^{f^\epsilon_{\ell,k}}$ is Lipschitz in the data $f^\epsilon_{\ell,k}$ with respect to the Hellinger metric. Specifically, if $\mu^{f^\epsilon_{\ell,k}}$ and $\mu^{{f_{\ell,k}^{\prime, \epsilon}}}$ are two measures corresponding to data $f^\epsilon_{\ell,k}$ and ${f_{\ell,k}^{\prime, \epsilon}}$ respectively, then for every $r>0$, there exists $C_{\ell,k}=C_{\ell,k}(r)>0$ such that for all $f^\epsilon_{\ell,k}$, ${f_{\ell,k}^{\prime, \epsilon}}\in \mathbb{R}^m$ with $\max\left\{\|f^\epsilon_{\ell,k}\|_{L^2(\Omega_0)},\|{f_{\ell,k}^{\prime, \epsilon}}\|_{L^2(\Omega_0)}\right\}<r$, we have $$d_{Hell}(\mu^{f^\epsilon_{\ell,k}},\mu^{{f_{\ell,k}^{\prime, \epsilon}}}) \leqslant C_{\ell,k} \|f^\epsilon_{\ell,k}-{f_{\ell,k}^{\prime, \epsilon}}\|_{\Sigma_{\ell,k}}.$$

	
\end{thm}
The above theorem is a direct consequence of the fact that $\Psi_{\ell,k}$ satisfies the Assumptions 2.6 in \textcolor{blue}{\cite{Stuart2010}}, and therefore we omit the proof here.

Combining equations \textcolor{blue}{(\ref{Vf})}, \textcolor{blue}{(\ref{varre1})}, \textcolor{blue}{(\ref{varim1})} and \textcolor{blue}{(\ref{covreim1})}, we obtain the following representation
\begin{equation}\label{exact_forward_F}
	T_{k,\ell_1,\ell_2}(\sigma^2)=F_{k,\ell_1,\ell_2},\ell_1=1,2,\ell_2=1,2,
\end{equation}
where
$$T_{k,\ell_1,\ell_2}(\sigma^2) = \left(T_{1,k,\ell_1,\ell_2,}(\sigma^2),T_{2,k,\ell_1,\ell_2,}(\sigma^2),\cdots,T_{k,\ell_1,\ell_2,}(\sigma^2)\right)^{\top}, F_{k,\ell_1,\ell_2} = \left(F_{1,k,\ell_1,\ell_2},F_{2,k,\ell_1,\ell_2},\cdots,F_{m,k,\ell_1,\ell_2}\right)^{\top},$$
with
\begin{equation}
	\begin{aligned}
		T_{j,k,\ell_1,\ell_2}(\sigma^2) =& Y_0(kr_{j,\ell_1})Y_0(kr_{j,\ell_2})\int_{\Omega_0}|\mathrm{Re}G_k(x,y)|^2\sigma^2(y)\mathrm{d}y+J_0(kr_{j,\ell_1})J_0(kr_{j,\ell_2})\int_{\Omega_0}|\mathrm{Im}G_k(x,y)|^2\sigma^2(y)\mathrm{d}y\\&-\left(Y_0(kr_{j,\ell_1})J_0(kr_{j,\ell_2})+Y_0(kr_{j,\ell_2})J_0(kr_{j,\ell_1})\right)\int_{\Omega_0}\mathrm{Re}G_k(x,y)\mathrm{Im}G_k(x,y)\sigma^2(y)\mathrm{d}y
	\end{aligned}
	\notag
\end{equation}
and
\begin{equation}
	\begin{aligned}
		F_{j,k,\ell_1,\ell_2} =4[\mathbf{Cov}(|v_{j,\ell_1,k}|^2,|v_{j,\ell_2,k}|^2)-\mathbf{Cov}(|u_{j,k}|^2,|v_{j,\ell_1,k}|^2)-\mathbf{Cov}(|u_{j,k}|^2,|v_{j,\ell_2,k}|^2)+\mathbf{Var}(|u_{j,k}|^2)].
	\end{aligned}
	\notag
\end{equation}
Hence the noisy measured data $F_{k,\ell_1,\ell_2}$ are generated by
\begin{equation}\label{err_forward_F}
	F^\epsilon_{k,\ell_1,\ell_2}=T_{k,\ell_1,\ell_2}(\sigma^2)+\eta_{k,\ell_1,\ell_2},\ \ell_1=1,2,\ \ell_2=1,2.
\end{equation}
where $\eta_{k,\ell_1,\ell_2} \thicksim \mathcal{N}(0,\Sigma_{k,\ell_1,\ell_2})$.
The forward operator $T_{k,\ell_1,\ell_2}:L^{2}(\Omega_0)\rightarrow \mathbb{R}^m$ for the inversion of $\sigma^2$ needs to satisfy the following lemma as in \textcolor{blue}{\cite{Stuart2010}}.
The proof follows a similar procedure as in Lemma \textcolor{blue}{\ref{forward operator assume1}}, thus we omit it here.
\begin{lem}\label{forward operator assume} \ \\
	(i) For every $\delta > 0$ there exists $M_{k,\ell_1,\ell_2}=M_{k,\ell_1,\ell_2}(\delta) \in \mathbb{R}$ such that, for all $\sigma^2\in L^{2}(\Omega_0)$,
	$$\|T_{k,\ell_1,\ell_2}(\sigma^2) \|_{\Sigma_{k,\ell_1,\ell_2}}\leqslant \exp (\delta \|\sigma^2\|^{2}_{L^2(\Omega_0)}+M_{k,\ell_1,\ell_2}),$$
	(ii)For every $ r>0$, there exists $K_{k,\ell_1,\ell_2}=K_{k,\ell_1,\ell_2}(r)>0$ such that, for all $\sigma^2_1,\sigma^2_2 \in L^2(\Omega_0)$ with $$\max\left\{\|\sigma^2_1\|_{L^2(\Omega_0)},\|\sigma^2_2\|_{L^2(\Omega_0)}\right\}<r,$$
	we have$$\big\|T_{k,\ell_1,\ell_2}(\sigma^2_1) - T_{k,\ell_1,\ell_2}(\sigma^2_2) \big\|_{\Sigma_{k,\ell_1,\ell_2}}\leqslant K_{k,\ell_1,\ell_2}\|\sigma^2_1-\sigma^2_2\|_{L^2(\Omega_0)}.$$
\end{lem}
Lemma \textcolor{blue}{\ref{forward operator assume}} imply that the forward operator $T_{k,\ell_1,\ell_2}$ is bounded and Lipschitz continuous, similar as Assumptions 2.6 in \textcolor{blue}{\cite{Stuart2010}}.
Consequently, the probability measure $\mu^{F^\epsilon_{k,\ell_1,\ell_2}}$ defined by equation \textcolor{blue}{\eqref{HybridPriorMeasure}} is well-defined on $L^2(\Omega_0)$ and is Lipschitz in the data $F^\epsilon_{k,\ell_1,\ell_2}$ with respect to the Hellinger metric, as shown in the following theorem.
\begin{thm}\label{welldefined}
	Let forward operator $T_{k,\ell_1,\ell_2}:L^2(\Omega_0)\rightarrow \mathbb{R}^m$ satisfies Lemma \ref{forward operator assume}. For a given $F^\epsilon_{k,\ell_1,\ell_2}\in \mathbb{R}^m$,  $\mu^{F^\epsilon_{k,\ell_1,\ell_2}}$ is given by equation (\ref{HybridPriorMeasure}). Then we have the following:\\
	$(i)$ $\mu^{F^\epsilon_{k,\ell_1,\ell_2}}$ defined as equation (\ref{HybridPriorMeasure}) is well defined on $L^2(\Omega_0)$.\\
	$(ii)$ $\mu^{F^\epsilon_{k,\ell_1,\ell_2}}$ is Lipschitz in the data $F^\epsilon_{k,\ell_1,\ell_2}$ with respect to the Hellinger metric. Specifically, if $\mu^{F^\epsilon_{k,\ell_1,\ell_2}}$ and $\mu^{{F_{k,\ell_1,\ell_2}^{\prime, \epsilon}}}$ are two measures corresponding to data $F^\epsilon_{k,\ell_1,\ell_2}$ and $F_{k,\ell_1,\ell_2}^\prime$ respectively, then for every $r>0$, there exists $C_{k,\ell_1,\ell_2}=C_{k,\ell_1,\ell_2}(r)>0$ such that $\max\left\{\|F^\epsilon_{k,\ell_1,\ell_2}\|_{L^2(\Omega_0)},\|{F_{k,\ell_1,\ell_2}^{\prime, \epsilon}}\|_{L^2(\Omega_0)}\right\}<r$ for all $F^\epsilon_{k,\ell_1,\ell_2}$, ${F_{k,\ell_1,\ell_2}^{\prime, \epsilon}}\in \mathbb{R}^m$, we have $$d_{Hell}(\mu^{F^\epsilon_{k,\ell_1,\ell_2}},\mu^{{F_{k,\ell_1,\ell_2}^{\prime, \epsilon}}}) \leqslant C_{k,\ell_1,\ell_2}\|F^\epsilon_{k,\ell_1,\ell_2}-{F_{k,\ell_1,\ell_2}^{\prime, \epsilon}}\|_{\Sigma}.$$
\end{thm}

\begin{rem}
Theorem \textcolor{blue}{\ref{welldefined1}} and Theorem \textcolor{blue}{\ref{welldefined}} ensure that the posterior measure (or posterior density) in the corresponding Bayesian formula is well-defined. In particular, the stability of the posterior measure with respect to the measurement data provides a theoretical foundation for the subsequent design of numerical algorithms.
\end{rem}

\section{Numerical examples}\label{5}
In the section, we present several numerical results of a two-dimensional example to evaluate the effectiveness of the proposed method. Next, we solve the IRSP using Bayesian inference methods.

First, the numerical approximation to the forward problem is obtained by discretizing the integral equations. We utilize Trapezoidal rules to approximate integrals, i.e.,
\[
\begin{aligned}
	I_1 &= \int_a^b \int_c^d g(y_1,y_2) \, dy_1 dy_2 \\&\approx \frac{\Delta y_1 \Delta y_2}{4} \sum_{i=0}^{N_{y_1}-1} \sum_{j=0}^{N_{y_2}-1} \big[
	g(y_1^{(i)}, y_2^{(j)}) + g(y_1^{(i+1)}, y_2^{(j)}) + g(y_1^{(i)}, y_2^{(j+1)}) + g(y_1^{(i+1)}, y_2^{(j+1)})
	\big],
\end{aligned}
\]
\[
\begin{aligned}
	I_2 = \int_a^b \int_c^d \sigma(y_1,y_2) \, \mathrm{d}W_y
	&\approx \frac{\Delta y_1 \Delta y_2}{4} \sum_{i=0}^{N_{y_1}-1} \sum_{j=0}^{N_{y_2}-1}  \big[
	\sigma(y_1^{(i)}, y_2^{(j)}) z_{i,j}
	+ \sigma(y_1^{(i+1)}, y_2^{(j)}) z_{i+1,j}
	\\&+ \sigma(y_1^{(i)}, y_2^{(j+1)}) z_{i,j+1}
	+ \sigma(y_1^{(i+1)}, y_2^{(j+1)}) z_{i+1,j+1}
	\big],\quad z_{ij} \sim N(0, \Delta t),
\end{aligned}
\]
where $ y_1^i = a + i \Delta y_1,i=0,1,\cdots,N_{y_1}$  and $ y_2^j = c + j \Delta y_2, j = 0,1,\cdots,N_{y_2}$ are the grid points in the \( y_1 \) and \( y_2 \) directions. $I_1$ is a deterministic integral and $I_2$ is a stochastic integral.
We producing $v_{j,\ell}$ by merely adding $u_{j,\ell}$ and $\Phi_{j,\ell}$ together. To evaluate the stability of the PR algorithm, we consider not only the discretization error of the forward solver but also add Gaussian distributed random noises to the measurements.
The noisy phaseless data is generated by the following formulas:
$$
\begin{aligned}
	\mathbf{E}^{\epsilon}(|u(x,k)|^2):&=(1+\epsilon r)\left(\mathbf{E}(|u(x,k)|^2)\right),\\
	\mathbf{Var}^{\epsilon}(|u(x,k)|^2):&=(1+\epsilon r)\left(\mathbf{Var}(|u(x,k)|^2)\right),\\
    \mathbf{Cov}^\epsilon(|v_{j,1,k}|^2,|v_{j,2,k}|^2):&=(1+\epsilon r)\left(\mathbf{Cov}(|v_{j,1,k}|^2,|v_{j,2,k}|^2)\right),\\
    \mathbf{E}^\epsilon(|v_{j,l,k}|^2):&=(1+\epsilon r)\left(\mathbf{E}(|v_{j,l,k}|^2))\right),\\
    \mathbf{Var}^\epsilon(|v_{j,l,k}|^2):&=(1+\epsilon r)\left(\mathbf{Var}(|v_{j,l,k}|^2))\right),\\
    \mathbf{Cov}^\epsilon(|u_{j,k}|^2,|v_{j,l,k}|^2):&=(1+\epsilon r)\left(\mathbf{Cov}(|u_{j,k}|^2,|v_{j,l,k}|^2)\right),\ \ \ l=1, 2,
\end{aligned}
$$
where $r$ is a  standard Gaussian distributed random number, and $\epsilon > 0$ is the noise level.

In our experiment, the test source function is chosen as the following

$$
g(x_{1},x_{2}) =0.3(1-x_{1})^{2}\mathrm{e}^{-x_{1}^{2}-(x_{2}+1)^{2}}-(0.2x_{1}-x_{1}^{3}-x_{2}^{5})\mathrm{e}^{-x_{1}^{2}-x_{2}^{2}}-0.03\mathrm{e}^{-(x_{1}+1)^{2}-x_{2}^{2}}$$
and

$$\sigma(x_{1},x_{2})=0.6\mathrm{e}^{-8(r^{3}-0.75r^{2})},\quad r=(x_{1}^{2}+x_{2}^{2})^{1/2},$$
and reconstruct the mean $g_1$ and the variance $\sigma_1$ given by   \\
$$g_1(x_1,x_2)=g(3x_1,3x_2)\quad\mathrm{and}\quad\sigma_1(x_1,x_2)=\sigma^2(x_1,x_2),$$\\
where $x_1,x_2 \in \Omega_0=[-1,1]\times[-1,1]$. See Figure \textcolor{blue}{\ref{figure_g1s1true}} for the surface plot of the exact $g_1$ and $\sigma_1$ in the domain $\Omega_0$.
\begin{figure}[H]
	\centering
	\subfigure[]{
		\label{g_value}
		\includegraphics[width=6cm]{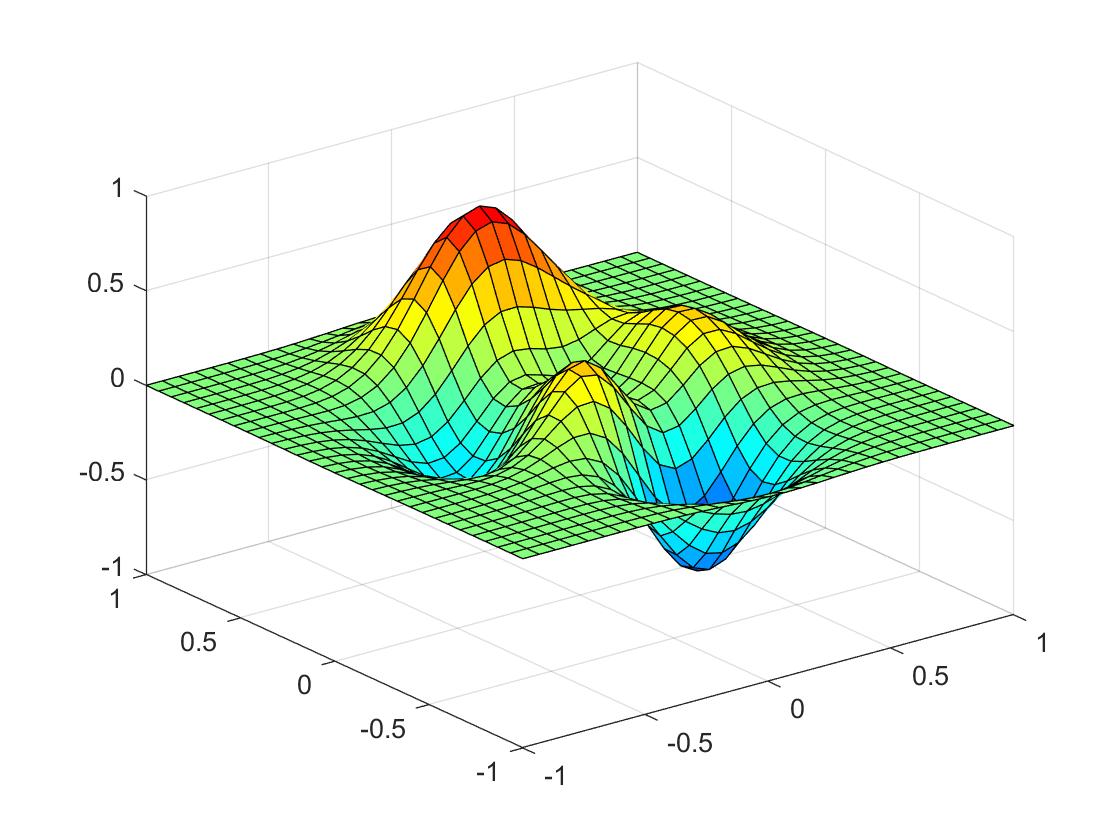}
	}
	\quad
	\subfigure[]{
		\label{s_value}
		\includegraphics[width=6cm]{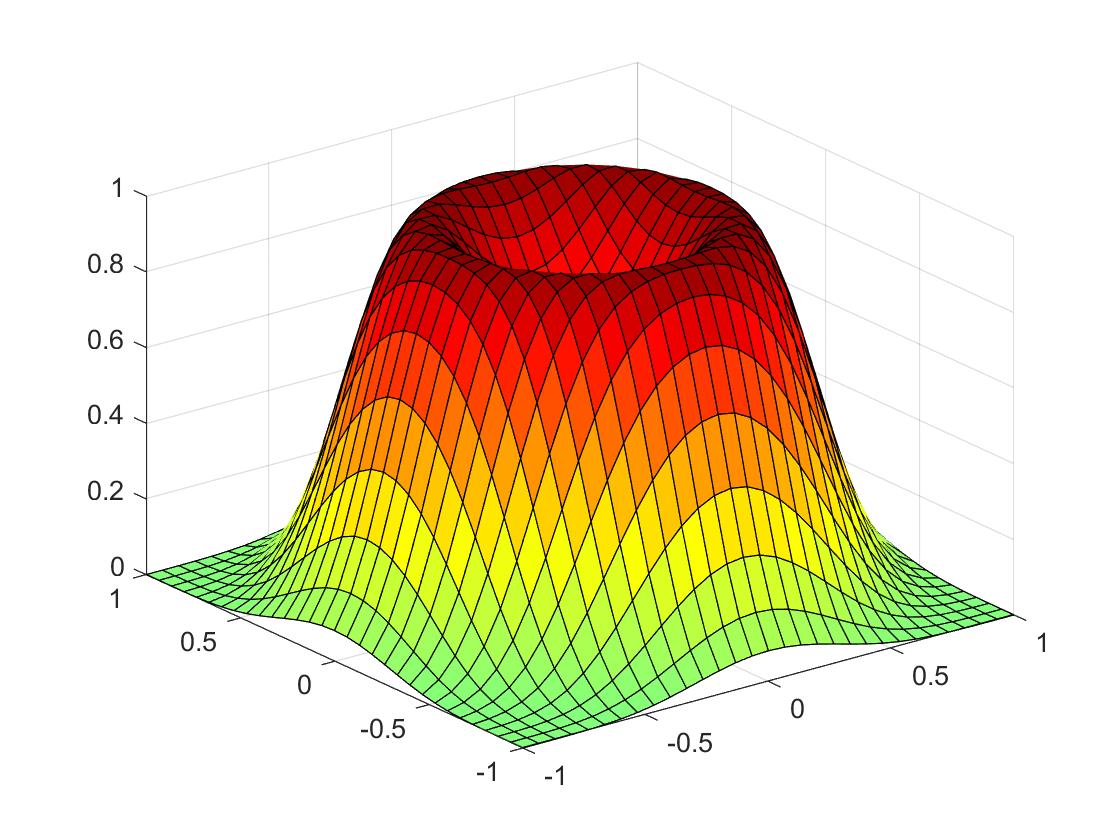}
	}
	\caption{(a) surface plot of the reconstructed mean $g_1$; (b) surface plot of the reconstructed variance $\sigma_1$.}
	\label{figure_g1s1true}
\end{figure}
\subsection{Phase retrieval}
In this section, a numerical validation of the phase retrieval technique will be shown.
First we specify the parameters used in \textcolor{blue}{\eqref{paramater}}: $a = 1,\tau  = 6, R = \tau a = 6,m= 10, \lambda_{1,1} = 0.5,\nu = 0.1 \pi$ and thus $k^* = \pi/30$. The other wavenumbers are taken as $k=\pi$, $k=60.5\pi$ and $k=84.5\pi$. We conducted experiments to confirm the accuracy and stability of the phase retrieval formulas. Table \textcolor{blue}{\ref{tab_Eerr}} and \textcolor{blue}{\ref{tab_varerr}} show the relative errors of phase retrieval for $\mathbf{E}(u)$ and $\mathbf{Var}(u)$ with different wavenumbers and noise levels. The discrete relative errors are calculated as follows:
$$\mathbf{E}(u)_{err} = \frac{\left(\sum_{j=1}^m\sum_{n=1}^{N_j}|\mathbf{E}(u(x_n,k))-\mathbf{E}^\epsilon(u(x_n,k))|^2\right)^{1/2}}{\left(\sum_{j=1}^m\sum_{n=1}^{N_j}\left|\mathbf{E}(u(x_n,k))\right|^2\right)^{1/2}},$$
$$\mathbf{Var}(u)_{err} = \frac{\left(\sum_{j=1}^m\sum_{n=1}^{N_j}|\mathbf{Var}(u(x_n,k))-\mathbf{Var}^\epsilon(u(x_n,k))|^2\right)^{1/2}}{\left(\sum_{j=1}^m\sum_{n=1}^{N_j}\left|\mathbf{Var}(u(x_n,k))\right|^2\right)^{1/2}},$$
$$\mathbf{Cov}(\operatorname{Re}(u),\operatorname{Im}(u))_{err} = \frac{\left(\sum_{j=1}^m\sum_{n=1}^{N_j}|\mathbf{Cov}(\operatorname{Re}(u(x_n,k)),\operatorname{Im}(u(x_n,k))-\mathbf{Cov}^\epsilon(\operatorname{Re}(u(x_n,k)),\operatorname{Im}(u(x_n,k))|^2\right)^{1/2}}{\left(\sum_{j=1}^m\sum_{n=1}^{N_j}\left|\mathbf{Cov}(\operatorname{Re}(u(x_n,k)),\operatorname{Im}(u(x_n,k))\right|^2\right)^{1/2}},$$
here, $x_n \in \partial B_j$ $(j=1,\dots,m)$, $N_j$ denotes the number of measurements on $\partial B_j$, $\mathbf{E}(u)$ and $\mathbf{E}^\epsilon(u)$ represent the exact and retrieved means of the radiated fields, $\mathbf{Var}(u)$ and $\mathbf{Var}^\epsilon(u)$ represent the exact and retrieved variances, $\mathbf{Cov}(\operatorname{Re}(u),\operatorname{Im}(u))$ and $\mathbf{Cov}^\epsilon(\operatorname{Re}(u),\operatorname{Im}(u))$ represent the exact and retrieved covariance.
The specifications are as follows: $N_j=120,j=1,\cdots,m.$ $N_{y_1}=60, N_{y_2}=60, \Delta t = 0.0022, M = 5\times10^4.$ Here, $M$ denotes the total number of samples in the pCN algorithm.

\begin{table}[H]
	\centering
	\caption{The relative $L^2$
		errors between $\mathbf{E}(u(\cdot,k))|_{\partial B_{R}}$ and $\mathbf{E}^\epsilon(
		u(\cdot,k))|_{\partial B_{R}}$ for different
		wavenumbers and noise levels}
	\begin{tabular}{ccc c c c c}
		\toprule
		& {$k^{*}=\frac{\pi}{30}$} & {$k=\pi$} &{$k=60.5\pi$}& {$k=84.5\pi$} \\
		\midrule
		$\epsilon = 0.01$ & 0.2034 & 0.1272 & 0.6575& 0.8143 \\
		$\epsilon = 0.005$ & 0.1014 & 0.0686 & 0.3585& 0.4411 \\
		$\epsilon = 0.001$ & 0.0247 & 0.0313 & 0.1872& 0.2287\\
		\bottomrule
	\end{tabular}
	\label{tab_Eerr}
\end{table}
\begin{table}[H]
\centering
  \caption{The relative $L^2$
 errors between $\mathbf{Var}(u(\cdot,k))|_{\partial B_{R}}$ and $\mathbf{Var}^\epsilon(
u(\cdot,k))|_{\partial B_{R}}$ for different
wavenumbers and noise levels}
    \begin{tabular}{ccc c c c}
    \toprule
    & {$k^{*}=\frac{\pi}{30}$} & {$k=\pi$} & {$k=60.5\pi$}&{$k=84.5\pi$} \\
    \midrule
    $\epsilon = 0.01$ & 0.0265 & 0.0122 &0.0140& 0.0141 \\
    $\epsilon = 0.005$ & 0.0135 & 0.0092 & 0.0113&0.0113 \\
    $\epsilon = 0.001$ & 0.0028 & 0.0082 & 0.0105& 0.0107\\
    \bottomrule
    \end{tabular}
    \label{tab_varerr}
\end{table}

\begin{table}[H]
\centering
  \caption{The relative $L^2$
 errors between $\mathbf{Cov}(\operatorname{Re}(u_{j,k}),\operatorname{Im}(u_{j,k}))|_{\partial B_{R}}$ and $\mathbf{Cov}^\epsilon(\operatorname{Re}(u_{j,k}),\operatorname{Im}(u_{j,k}))|_{\partial B_{R}}$ for different
wavenumbers and noise levels}
    \begin{tabular}{ccc c c c}
    \toprule
    & {$k^{*}=\frac{\pi}{30}$} & {$k=\pi$} & {$k=60.5\pi$}&{$k=84.5\pi$} \\
    \midrule
    $\epsilon = 0.01$ & 0.0484 & 0.0620 &0.3256& 0.2438 \\
    $\epsilon = 0.005$ & 0.0247 & 0.0482 & 0.2352&0.1696 \\
    $\epsilon = 0.001$ & 0.0049 & 0.0429 & 0.1870& 0.1363\\
    \bottomrule
    \end{tabular}
    \label{tab_coverr}
\end{table}
From Table \textcolor{blue}{\ref{tab_Eerr}},
\textcolor{blue}{\ref{tab_varerr}} and  \textcolor{blue}{\ref{tab_coverr}}, the PR algorithm can accurately recover the phase information at different wavenumbers in small noise level. As we expect, the relative errors gradually decrease as the noise level decreases when the wavenumber is fixed. The conclusions of Theorems \textcolor{blue}{\ref{Esta}} and \textcolor{blue}{\ref{var_stab}} can be verified by these results. This demonstrates the effectiveness and feasibility of PR algorithm.
\subsection{Inverse random source}
In this section, we present the numerical results of reconstructing the random source using the Bayesian method.

We can reconstruct the mean $g_1$ via the equations \textcolor{blue}{\eqref{err_forward_f}} for a set of wavenumbers $k \in \{k_1,k_2,\cdots,k_N\}$. Similarly, the variance $\sigma_1$ can be reconstructed using the equations \textcolor{blue}{\eqref{err_forward_F}} for the same set of wavenumbers. In the following examples, we take $k_j = (j+82.5)\pi,j = 0, 1, 2, 3$.
We can simply discretize equation \textcolor{blue}{\eqref{err_forward_f}} and \textcolor{blue}{\eqref{err_forward_F}} to obtain a discrete linear system by using Trapezoidal rules on a uniform grid.

Specifically, we choose the covariance operator $\mathcal{C}_{0}$ of the Gaussian prior $\mathcal{N}(0,\mathcal{C}_{0})$ is given by
\begin{equation}\label{c_0}
	c_{0}(x_1,x_2)=\gamma\exp\left[-\frac{1}{2}\left(\frac{|x_1-x_2|}{d}\right)^2\right],
\end{equation}
where $\gamma=0.001$ and $d=0.2$ in the subsequent numerical experiment of reconstructing $g_1$. Additionally, we fix $\beta = 0.05$ and draw $10 ^ 5 $ samples from the posterior distribution using pCN algorithm, with a burn-in period of $ 5 \times 10 ^ 3 $.
The error of the numerical experiments is measured using relative $L^2$ error
$$g_{1_{err}} = \frac{\sqrt{\sum_{i=1}^{N^{'}}|\hat{g}_{1_{i}}-g_{1_{i}}|^2}}{\sqrt{\sum_{i=1}^{N^{'}}|g_{1_{i}}|^2}}.$$
where $g_1$ is the true solution and $\hat g_1$ is the reconstructed result obtained using our algorithm. When solving inverse problems, we employ a coarse grid to avoid the ``inverse crimes''
, i.e., $N_j=60,\ j=1,\cdots,m$. $N_{y_1}=30$, $N_{y_2}=30.$

The numerical results for reconstructing the source term $g_1$ are shown in Figure \textcolor{blue}{\ref{figure_g1_82.5}}. For $N=4$, the inversion attains a satisfactory outcome. The inversion results using $f^\epsilon_{1,k}$ or $f^\epsilon_{2,k}$ are similar.
\begin{figure}[H]
	\centering
	\subfigure[]{
		\label{82.5RE5}
		\includegraphics[width=6cm]{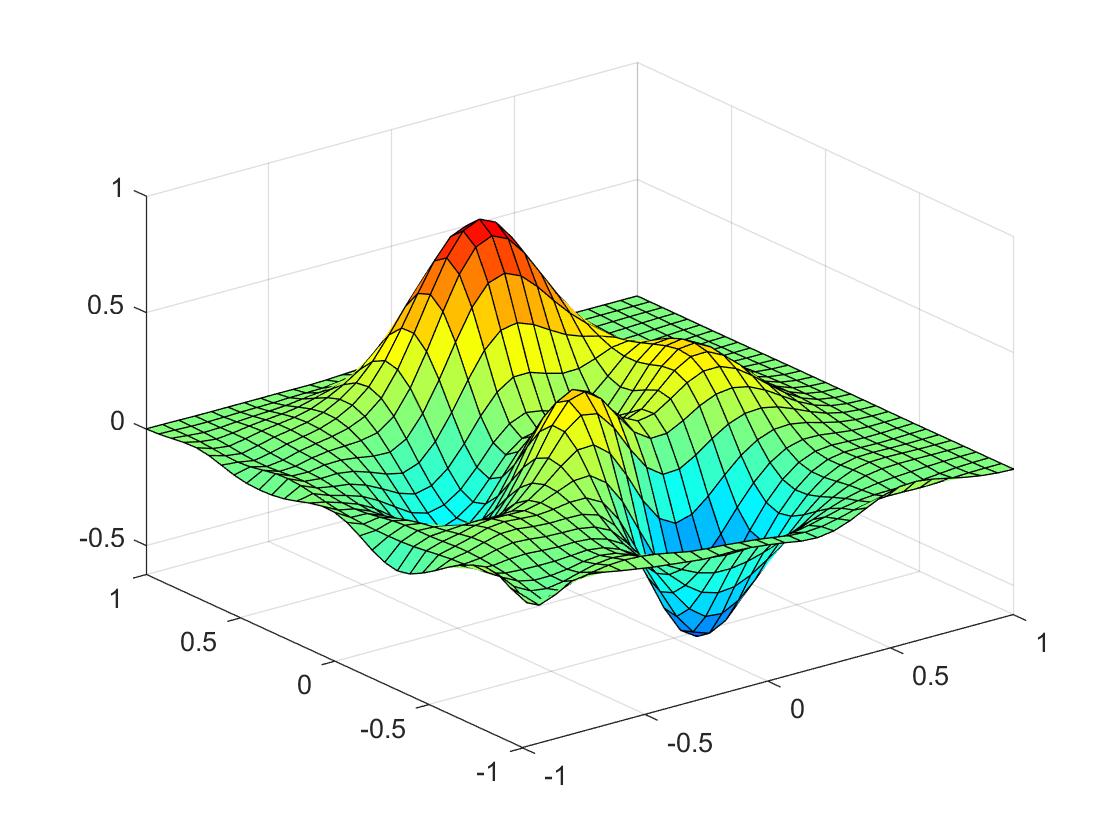}
	}
	\quad
	\subfigure[]{
		\label{82.5IM5}
		\includegraphics[width=6cm]{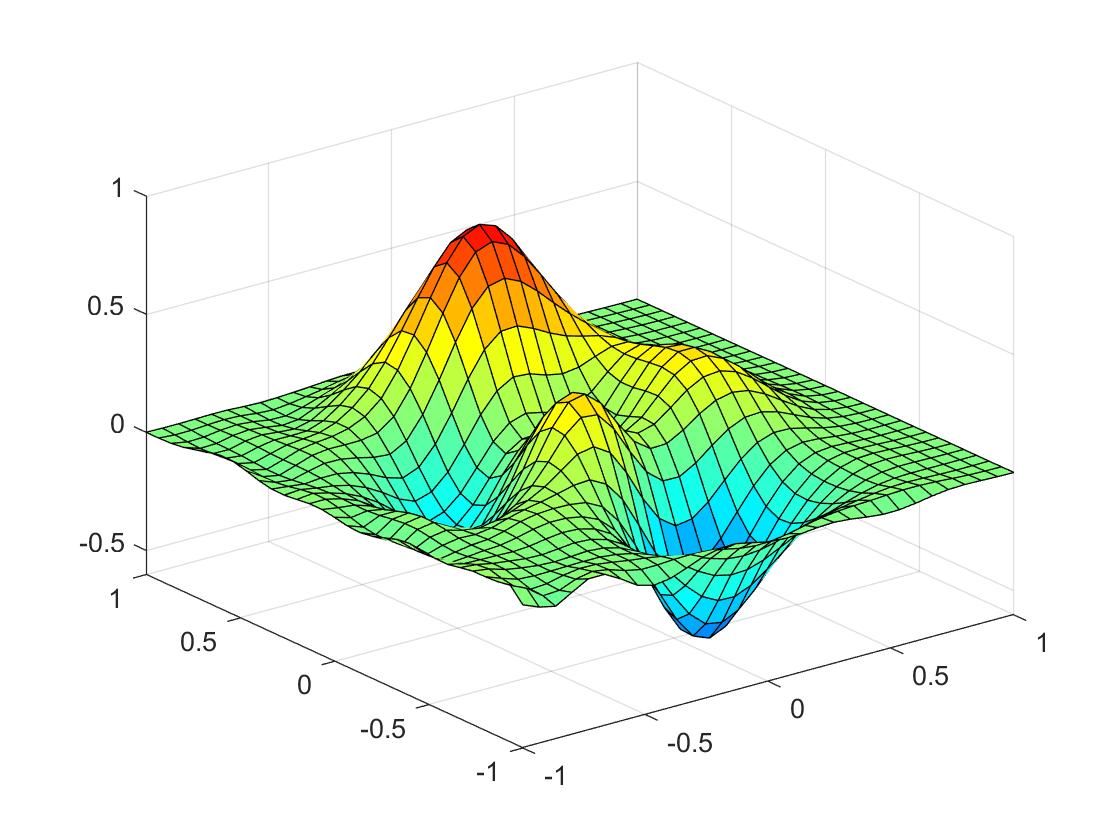}
	}
	\caption{Surface plots of the reconstructed mean $g_1$ using different equations with $\epsilon = 0.1\%$. (a) $f^\epsilon_{1,k}$. (b) $f^\epsilon_{2,k}$.}
	\label{figure_g1_82.5}
\end{figure}
The inversion process for $\sigma_1$ is the same as that for $g_1$.
When we use $F^\epsilon_{k,1,1}$ or $F^\epsilon_{k,2,2}$ to reconstruct $\sigma_1$, we take $\gamma=0.01,d = 0.2,\beta=0.01$ and draw $10 ^ 5 $ samples from pCN algorithm with $ 5 \times 10 ^ 3 $ in burn-in period. When we use $F^\epsilon_{k,1,2}$, we change paramater $\gamma$ to 0.05.
From Table \textcolor{blue}{\ref{tab_gserr}} and Figure \textcolor{blue}{\ref{figure_s1_82.5}}, we can accurately recover the true solution in all cases.

\begin{figure}[H]
	\centering
	\subfigure[]{
		\label{ff1}
		\includegraphics[width=0.3\textwidth]{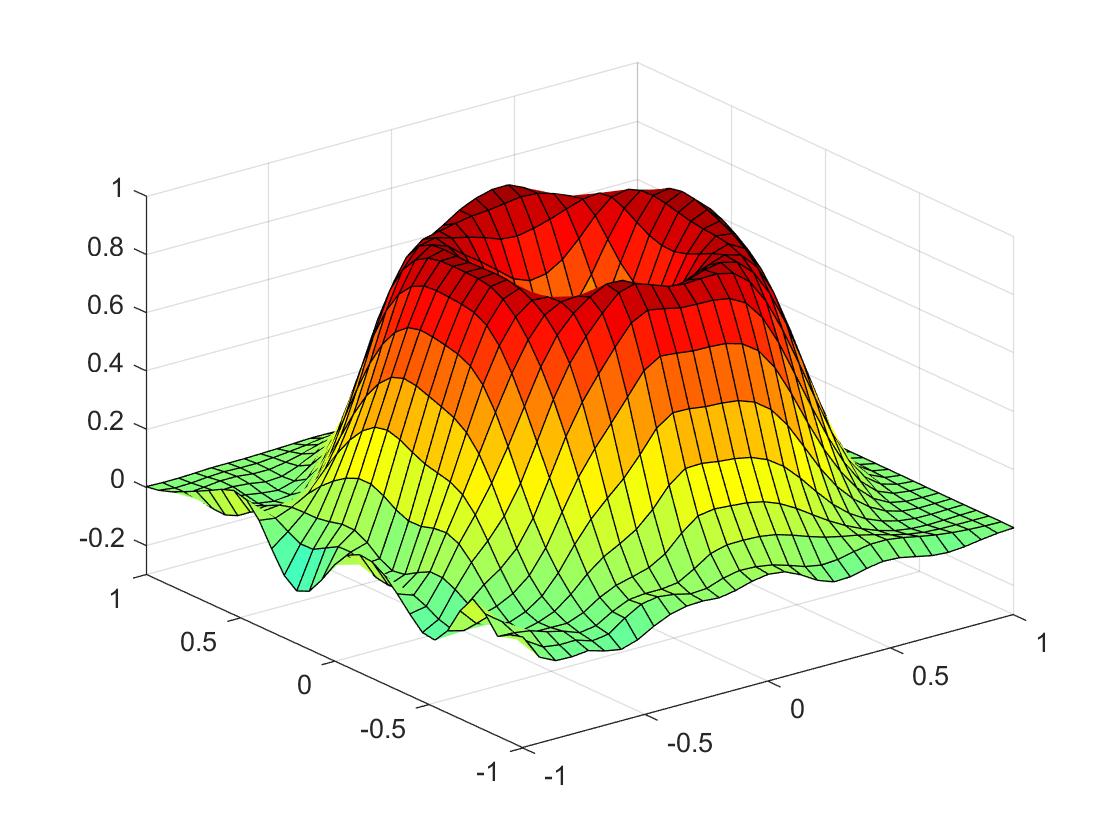}
	}
	\quad
	\subfigure[]{
		\label{ff2}
		\includegraphics[width=0.3\textwidth]{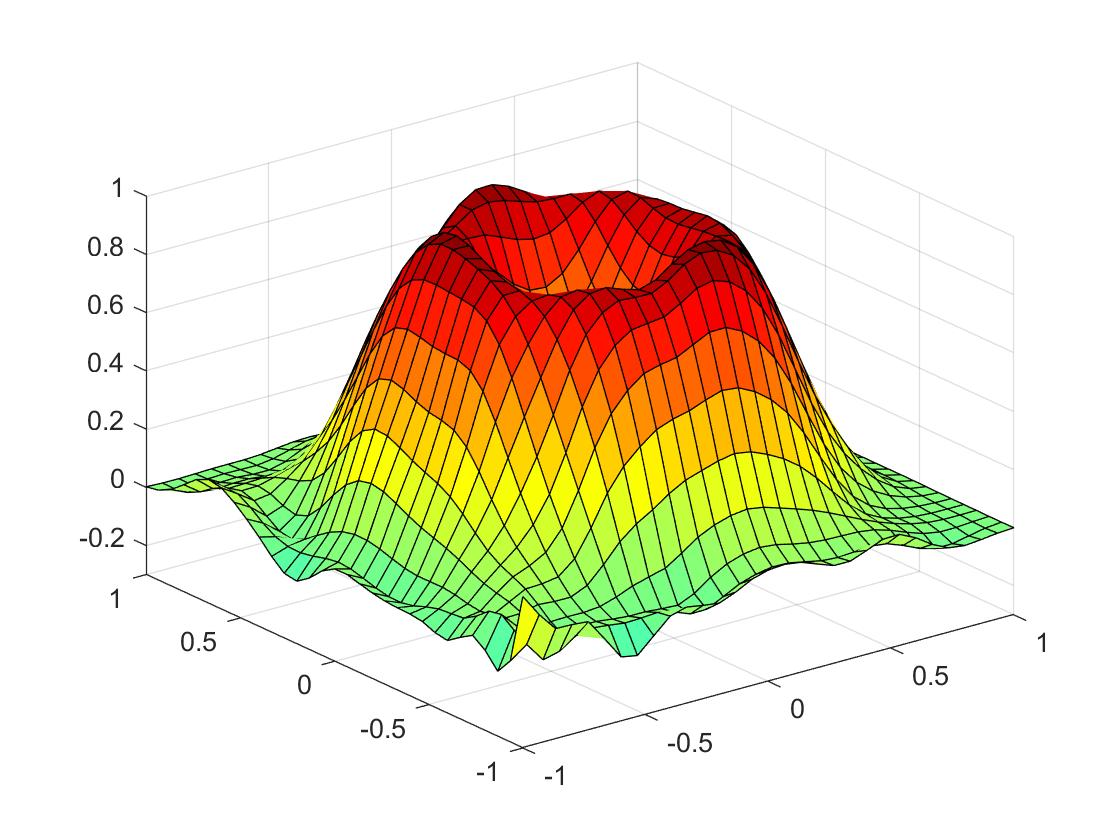}
	}
	\quad
	\subfigure[]{
		\label{ff3}
		\includegraphics[width=0.3\textwidth]{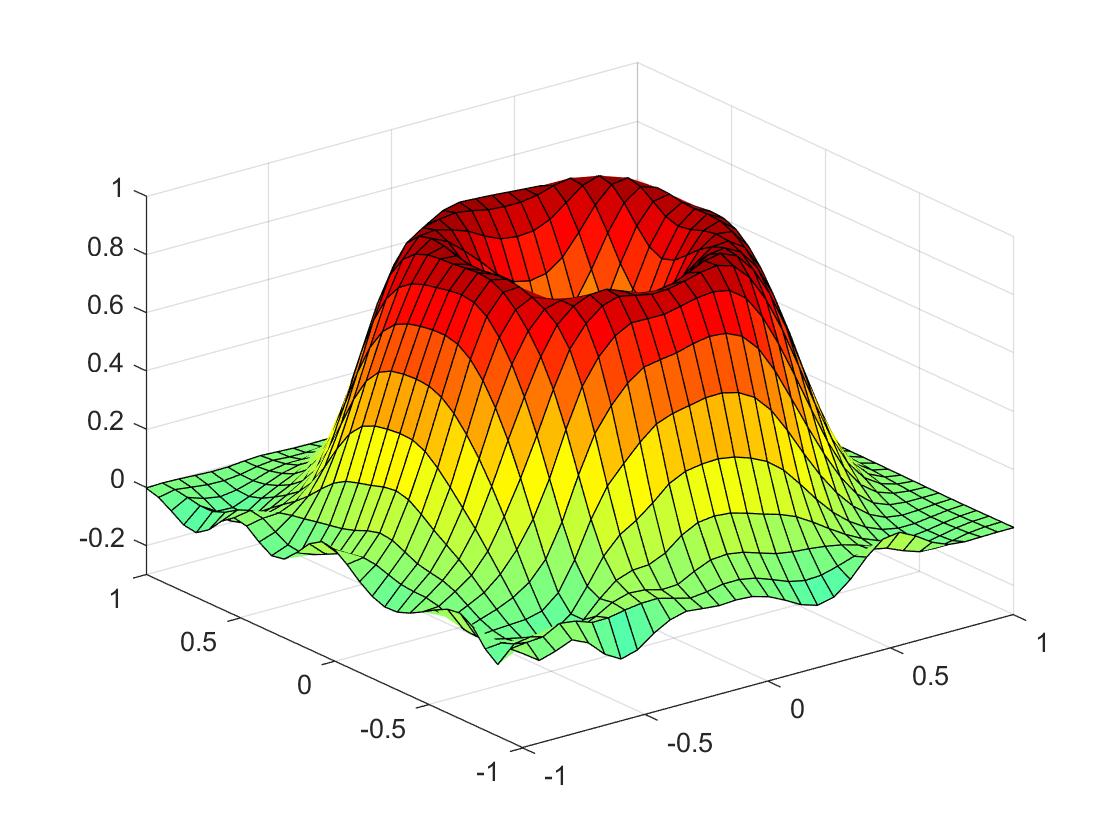}
	}
	\caption{Surface plots of the reconstructed mean $\sigma_1$ using different equations with $\epsilon = 0.1\%$. (a) $F^\epsilon_{k,1,1}$. (b) $F^\epsilon_{k,2,2}$.(c) $F^\epsilon_{k,1,2}$.}
	\label{figure_s1_82.5}
\end{figure}

\begin{table}[H]
	\centering
	\caption{The relative $L^2$ errors with various noise levels $\epsilon$.}
	\begin{tabular}{ccccc}
		\toprule
		& & $\epsilon=5\%$ & $\epsilon=1\%$ & $\epsilon=0.1\%$ \\
		\midrule
		$g_{1_{err}}$ & $f^\epsilon_{1,k}$ & 0.5954 & 0.1499 & 0.1211 \\
		& $f^\epsilon_{2,k}$ & 0.5773 & 0.1475 & 0.1151 \\
		\midrule
		$\sigma_{1_{err}}$ & $F^\epsilon_{k,1,1}$ & 0.3950 & 0.1102 & 0.1030 \\
		& $F^\epsilon_{k,2,2}$ & 0.3541 & 0.1183 & 0.1077 \\
		& $F^\epsilon_{k,1,2}$ & 0.1761 & 0.0832 & 0.0830 \\
		\bottomrule
	\end{tabular}
	\label{tab_gserr}
\end{table}

In Table \textcolor{blue}{\ref{tab_gserr}}, the reconstruction of $g_1$ becomes more accurate as the noise level decreases. Note that even with low levels of noise, the reconstruction remains quite precise and stable, regardless of whether $f^\epsilon_{1,k}$ or $f^\epsilon_{2,k}$ is used. However, when the noise level is high, the inversion results deteriorate significantly. When reconstructing $\sigma_1$, similar results are obtained.
\section{Conclusion}\label{6}
In this paper, we have investigated methods to determine acoustic random sources from multi-frequency phaseless data. The reconstruction process consists of two stages. On the first stage, phase retrieval formulas for the mean and variance of the scattered field are derived using the reference source technique, and the uniqueness and stability of the retrieved statistics are proved. On the second stage, based on the mild solution of the Helmholtz equation and phase retrieval formulas, we derive Fredholm integral equations for the inverse scattering problem to reconstruct the mean and variance of the random source. The Bayesian method is employed to solve the inverse problem for the reconstruction of random sources. The stability of the integral equations problems are proved, thereby leading to the stability results for the phaseless IRSP. A two-dimensional numerical example is provided to demonstrate the effectiveness and validity of the proposed method. The proposed framework and methodology can be directly applied to solving a wide range of IRSP with phaseless data, such as the heat and wave equations. Moreover, the methodology can be extended to random source problems in inhomogeneous media.

\section*{Acknowledgements}
\addcontentsline{toc}{section}{Acknowledgments}
The research of H. Liu was supported by NSFC/RGC Joint Research Scheme, N CityU101/21, ANR/RGC Joint Research Scheme, A-CityU203/19, and the Hong Kong RGC General Research Funds (projects 11311122, 11303125 and 11300821). The research of G. Zheng was supported by the NSF of China (12271151).

\section{Appendix}

\subsection{Appendix A}
\label{pf}

\begin{lem}\label{lem_Wx}
\text{Let $f(x)$ be a test function with a compact support in $\mathbb{R}^d.$ It holds that}\\
\begin{equation}\label{onefunc}
	\mathbf{E}\left(\int_{\mathbb{R}^d}f(x)\mathrm{d}{W}_x\right)=0,\quad\mathbf{E}\left(\left|\int_{\mathbb{R}^d}f(x)\mathrm{d}{W}_x\right|^2\right)=\int_{\mathbb{R}^d}|f(x)|^2\mathrm{d}x.
\end{equation}
\text{If $g(x)$ is also a fucntion with a compact support in $\mathbb{R}^d$, it holds that}\\
\begin{equation}\label{twofunc}
	\mathbf{E}\left(\int_{\mathbb{R}^d}f(x)\mathrm{d}{W}_x\cdot\int_{\mathbb{R}^d}g(x)\mathrm{d}{W}_x\right)=\int_{\mathbb{R}^d}f(x)\cdot g(x)\mathrm{d}x.
\end{equation}
\end{lem}
\begin{proof}
The proof of \textcolor{blue}{\eqref{onefunc}} can be found in \textcolor{blue}{\cite{Bao2016}}.
Now, we prove the third equation.
According to \textcolor{blue}{\cite{Bao2016}}, we have
\begin{equation}\label{A.1}
	\mathbf{E}({W}_x{W}_y)=x\wedge y:=(x_1\wedge y_1)\cdots(x_d\wedge y_d)
\end{equation}
for any $x=(x_1,\ldots,x_d)\in\mathbb{R}_+^d\mathrm{~and~}y=(y_1,\ldots,y_d)\in\mathbb{R}_+^d,\mathrm{~where~}x_j\wedge y_j=\min\{x_j,y_j\}.$
We also have
\begin{equation}\label{A.2}
	\int_{\mathbb{R}^{d}}f(x)\mathrm{d}{W}_{x} = (-1)^d\int_{\mathbb{R}^{d}}{W}_{x}\frac{\partial^{d}f(x)}{\partial x_{1}\cdots\partial x_{d}}\mathrm{d}x
\end{equation}
and
\begin{equation}\label{A.3}
	\quad\int_{\mathbb{R}^d}\frac{\partial^d f(x)}{\partial x_1\cdots\partial x_d}(x\wedge y)\mathrm{d}x=(-1)^d\int_{-\infty}^{y_1}\cdots\int_{-\infty}^{y_d}f(x)\mathrm{d}x_d\cdots\mathrm{d}x_1.
\end{equation}
From \textcolor{blue}{\eqref{A.1}}, \textcolor{blue}{\eqref{A.2}} and \textcolor{blue}{\eqref{A.3}}, it follows that
\begin{equation}\label{A.4}
	\begin{aligned}
		&\mathbf{E}\left(\int_{\mathbb{R}^{d}}f(x)\mathrm{d}{W}_{x}\cdot\int_{\mathbb{R}^{d}}g(y)\mathrm{d}{W}_{y}\right)\\
		=&\mathbf{E}\left(\int_{\mathbb{R}^{d}}{W}_{x}\frac{\partial^{d}f(x)}{\partial x_{1}\cdots\partial x_{d}}\mathrm{d}x\cdot\int_{\mathbb{R}^{d}}{W}_{y}\frac{\partial^{d}g(y)}{\partial y_{1}\cdots\partial y_{d}}\mathrm{d}y\right) \\
		=&\int_{\mathbb{R}^{d}}\int_{\mathbb{R}^{d}}\frac{\partial^{d}f(x)}{\partial x_{1}\cdots\partial x_{d}}\frac{\partial^{d}g(y)}{\partial y_{1}\cdots\partial y_{d}}\mathbf{E}({W}_{x}{W}_{y})\mathrm{d}x\mathrm{d}y \\
		=&\int_{\mathbb{R}^{d}}\frac{\partial^{d}g(y)}{\partial y_{1}\cdots\partial y_{d}}\left(\int_{\mathbb{R}^{d}}\frac{\partial^{d}f(x)}{\partial x_{1}\cdots\partial x_{d}}(x\wedge y)\mathrm{d}x\right)\mathrm{d}y\\
		=&\int_{\mathbb{R}^{d}}\frac{\partial^{d}g(y)}{\partial y_{1}\cdots\partial y_{d}}\left((-1)^d\int_{-\infty}^{y_1}\cdots\int_{-\infty}^{y_d}f(x)\mathrm{d}x_d\cdots\mathrm{d}x_1\right)\mathrm{d}y\\
		=&(-1)^d\int_{-\infty}^{+\infty}\cdots\int_{-\infty}^{+\infty}\frac{\partial^{d}g(y)}{\partial y_{1}\cdots\partial y_{d}}\left(\int_{-\infty}^{y_1}\cdots\int_{-\infty}^{y_d}f(x)\mathrm{d}x_d\cdots\mathrm{d}x_1\right)\mathrm{d}y.
	\end{aligned}
\end{equation}

First, apply integration by parts to $y_1$. Since $g(y)$ has a compact support in $\mathbb{R}^d$, we have
$$\begin{aligned}&\int_{-\infty}^{+\infty}\left[\frac{\partial}{\partial y_{1}}(\frac{\partial^{d-1}g(y_{1},y_{2},\cdots y_{d})}{\partial y_{2}\cdots\partial y_{d}})\int_{-\infty}^{y_{1}}f(x)\mathrm{d}x_{1}\right]\mathrm{d}y_{1}\\=&\frac{\partial^{d-1}g(y_{1},y_{2}\cdots y_{d})}{\partial y_{2}\cdots\partial y_{d}}\cdot\int_{-\infty}^{y_{1}}f(x)\mathrm{d}x_{1}|_{-\infty}^{+\infty}-\int_{-\infty}^{+\infty}\frac{\partial^{d-1}g(y_{1},\cdots y_{d})}{\partial y_{2}\cdots\partial y_{d}}\cdot f(y_{1},x_{2},\cdots,x_{d})\mathrm{d}y_{1}\\
	=&-\int_{-\infty}^{+\infty}\frac{\partial^{d-1}g(y_{1},\cdots y_{d})}{\partial y_{2}\cdots\partial y_{d}}\cdot f(y_{1},x_{2},\cdots,x_{d})\mathrm{d}y_{1}.
\end{aligned}$$
Then for $y_2$, we also obtain
$$\begin{aligned}&\int_{-\infty}^{+\infty}\frac{\partial^{d-1}g(y_{1},\cdots y_{d})}{\partial y_{2}\cdots\partial y_{d}}\left(\int_{-\infty}^{y_{2}}f(y_{1},x_{2},\cdots x_{d})\mathrm{d}x_{2}\right)\mathrm{d}y_{2}\\=&\int_{-\infty}^{+\infty}\frac{\partial}{\partial y_{2}}\left(\frac{\partial^{d-2}(y_{1},\cdots y_{d})}{\partial y_{3}\cdots\partial y_{d}}\right)\left(\int_{-\infty}^{y_{2}}f(y_{1},x_{2},\cdots x_{d})\mathrm{d}x_{2}\right)\mathrm{d}y_{2}\\=&\frac{\partial^{d-2}g(y_{1},y_{2}\cdots y_{d})}{\partial y_{3}\cdots\partial y_{d}}\int_{-\infty}^{y_{2}}f(y_{1},x_{2},\cdots x_{d})\mathrm{d}x_{2}|_{-\infty}^{+\infty}-\int_{-\infty}^{+\infty}\frac{\partial^{d-2}g(y_{1},y_{2},\cdots,y_{d})}{\partial y_{3}\cdots\partial y_{d}}\cdot f(y_{1},y_{2},x_{2}\cdots x_{d})\mathrm{d}y_{2}\\=&-\int_{-\infty}^{+\infty}\frac{\partial^{d-2}g(y_{1},y_{2},\cdots y_{d})}{\partial y_{2}\cdots\partial y_{d}}f(y_{1},y_{2},x_{3}\cdots x_{d})\mathrm{d}y_{2}.\end{aligned}$$
Repeat the process for $y_3,\cdots,y_d$, hence
\begin{equation}
	\begin{aligned}
		\mathbf{E}\left(\int_{\mathbb{R}^{d}}f(x)\mathrm{d}{W}_{x}\cdot\int_{\mathbb{R}^{d}}g(y)\mathrm{d}{W}_{y}\right)&= (-1)^d\cdot(-1)^d\int_{-\infty}^{+\infty}\cdots\int_{-\infty}^{+\infty}g(y_1,\cdots,y_d)f(y_1,\cdots,y_d)\mathrm{d}y_1\cdots \mathrm{d}y_d\\&=
		\int_{\mathbb{R}^{d}}f(x)g(x)\mathrm{d}x,
	\end{aligned}
\end{equation}
which completes the proof.
\end{proof}

\end{document}